\documentclass[final,subeqn]{siamltex}
\usepackage{amsfonts,amssymb,calc,epsfig,mathrsfs,multicol,multirow}

\newcommand{\E}{{\tiny E}}
\newcommand{\FMM}{\mathrm{FMM}}
\newcommand{\IMQ}{\mathrm{IMQ}}
\newcommand{\MATLAB}{{\sc Matlab}}
\newcommand{\MQ}{\mathrm{MQ}}
\newcommand{\TPS}{\mathrm{TPS}}
\newcommand{\argmin}{\mathop{\arg\,\min}}
\newcommand{\bigO}{\mathcal{O}}
\newcommand{\child}{\mathop{\mathrm{child}}}
\newcommand{\cm}{\mathrm{cm}}
\newcommand{\col}{\mathrm{c}}
\newcommand{\etal}{et al.~}
\newcommand{\iter}{\mathrm{iter}}
\newcommand{\mach}{\mathrm{mach}}
\newcommand{\qr}{\mathrm{qr}}
\newcommand{\range}{\mathop{\mathrm{range}}}
\newcommand{\row}{\mathrm{r}}
\newcommand{\sv}{\mathrm{sv}}
\newcommand{\trans}{\mathsf{T}}

\title{A fast semi-direct least squares algorithm for hierarchically block separable matrices\thanks{This work was supported in part by the National Science Foundation under awards DGE-0333389 and DMS-1203554, by the U.S.\ Department of Energy under contract DEFG0288ER25053, and by the Air Force Office of Scientific Research under NSSEFF Program Award FA9550-10-1-0180.}}
\author{Kenneth L. Ho\thanks{Courant Institute of Mathematical Sciences and Program in Computational Biology, New York University, New York, NY. Present address: Department of Mathematics, Stanford University, Stanford, CA ({\tt klho@stanford.edu}).} \and Leslie Greengard\thanks{Courant Institute of Mathematical Sciences, New York University, New York, NY ({\tt greengard@courant.nyu.edu}).}}

\begin{document}

 \maketitle

 \begin{abstract}
  We present a fast algorithm for linear least squares problems governed by hierarchically block separable (HBS) matrices.  Such matrices are generally dense but {\em data-sparse} and can describe many important operators including those derived from asymptotically smooth radial kernels that are not too oscillatory. The algorithm is based on a recursive skeletonization procedure that exposes this sparsity and solves the dense least squares problem as a larger, equality-constrained, sparse one. It relies on a sparse QR factorization coupled with iterative weighted least squares methods. In essence, our scheme consists of a direct component, comprised of matrix compression and factorization, followed by an iterative component to enforce certain equality constraints. At most two iterations are typically required for problems that are not too ill-conditioned. For an $M \times N$ HBS matrix with $M \geq N$ having bounded off-diagonal block rank, the algorithm has optimal $\bigO (M + N)$ complexity. If the rank increases with the spatial dimension as is common for operators that are singular at the origin, then this becomes $\bigO (M + N)$ in 1D, $\bigO (M +  N^{3/2})$ in 2D, and $\bigO (M + N^{2})$ in 3D. We illustrate the performance of the method on both over- and underdetermined systems in a variety of settings, with an emphasis on radial basis function approximation and efficient updating and downdating.
 \end{abstract}



 \pagestyle{myheadings}
 \thispagestyle{plain}
 \markboth{K.\ L.\ HO AND L.\ GREENGARD}{FAST LEAST SQUARES FOR HIERARCHICAL MATRICES}

 \section{Introduction}
 The method of least squares is a powerful technique for the approximate solution of overdetermined systems and is often used for data fitting and statistical inference in applied science and engineering. In this paper, we will primarily consider the linear least squares problem
 \begin{equation}
  \min_{x} \| Ax - b \|,
  \label{eqn:overdetermined}
 \end{equation}
 where $A \in \mathbb{C}^{M \times N}$ is dense and full-rank with $M \geq N$, $x \in \mathbb{C}^{N}$, $b \in \mathbb{C}^{M}$, and $\| \cdot \|$ is the Euclidean norm. Formally, the solution is given by
 \begin{equation}
  x = A^{+} b,
  \label{eqn:solution}
 \end{equation}
 where $A^{+}$ is the Moore-Penrose pseudoinverse of $A$, and can be computed directly via the QR decomposition at a cost of $\bigO (M N^{2})$ operations \cite{bjorck:1996:siam,lawson:1974:prentice-hall}. This can be prohibitive when $M$ and $N$ are large. If $A$ is structured so as to support fast multiplication, then iterative methods such as LSQR \cite{paige:1982:acm-trans-math-softw} or GMRES \cite{hayami:2010:siam-j-matrix-anal-appl,saad:1986:siam-j-sci-stat-comput} present an attractive alternative. However, such solvers still have several key disadvantages when compared with their direct counterparts:
 \begin{romannum}
  \item
   The convergence rate of an iterative solver can depend strongly on the conditioning of the system matrix, which, for least squares problems, can sometimes be very poor. In such cases, the number of iterations required, and hence the computational cost, can be far greater than expected (if the solver succeeds at all). Direct methods, by contrast, are robust in that their performance does not degrade with conditioning. Thus, they are often preferred in situations where reliability is critical.
  \item
   Standard iterative schemes are inefficient for multiple right-hand sides. With direct solvers, on the other hand, following an expensive initial factorization, the subsequent cost for each solve is typically much lower (e.g., only $\bigO (MN)$ work to apply the pseudoinverse given precomputed QR factors). This is especially important in the context of updating and downdating as the least squares problem is modified by adding or deleting data, which can be viewed as low-rank updates of the original system matrix.
 \end{romannum}

 In this paper, we present a fast semi-direct least squares solver for a class of structured dense matrices called hierarchically block separable (HBS) matrices. Such matrices were introduced by Gillman, Young, and Martinsson \cite{gillman:2012:front-math-china} and possess a nested low-rank property that enables highly efficient data-sparse representations. The HBS matrix structure is closely related to that of $\mathscr{H}$- and $\mathscr{H}^{2}$-matrices \cite{hackbusch:1999:computing,hackbusch:2002:computing,hackbusch:2000:computing,hackbusch:2000:lect-appl-math} and hierarchically semiseparable (HSS) matrices \cite{chandrasekaran:2006a:siam-j-matrix-anal-appl,chandrasekaran:2006b:siam-j-matrix-anal-appl,xia:2010:numer-linear-algebra-appl}, and can be considered a generalization of the matrix features utilized by multilevel summation algorithms like the fast multipole method (FMM) \cite{greengard:1987:j-comput-phys,greengard:1997:acta-numer}. Many linear operators are of HBS form, notably integral transforms with asymptotically smooth radial kernels. This includes those based on the Green's functions of non-oscillatory elliptic partial differential equations \cite{borm:2007:linear-algebra-appl}. Some examples are shown in Table \ref{tab:examples}; we highlight, in particular, the Green's functions
 $$
  \phi_{\Delta} \left( r \right) = \frac{1}{r}, \quad \phi_{\Delta^{2}} \left( r \right) = r,
 $$
 for the Laplace and biharmonic equations, respectively, in 3D, and their regularizations, the inverse multiquadric and multiquadric kernels
 $$
  \phi_{\IMQ} \left( r \right) = \frac{1}{\sqrt{r^{2} + c^{2}}}, \quad \phi_{\MQ} \left( r \right) = \sqrt{r^{2} + c^{2}},
 $$
 respectively (for $c$ not too large).
 \begin{table}
  \caption{Examples of radial kernels $\phi (r)$ admitting HBS representations: $H_{0}^{(1)}$, zeroth order Hankel function of the first kind; $K_{0}$, zeroth order modified Bessel function of the second kind.}
  \label{tab:examples}
  \begin{center}
   \footnotesize
   \begin{tabular}{llccl}
    \hline
    Type & Name & \multicolumn{2}{c}{Kernel} & Notes\\
    \hline
    & & 2D & 3D\\
    \multirow{4}{*}{Green's function} & Laplace & $\log r$ & $1/r$\\
    & Helmholtz & $H_{0}^{(1)} (kr)$ & $e^{\imath kr} / r$ & $k$ not too large\\
    & Yukawa & $K_{0} (kr)$ & $e^{-kr} / r$\\
    & Polyharmonic & $r^{2n} \log r$ & $r^{2n - 1}$ & $n = 1, 2, 3, \dots$\\
    \hline
    \multirow{2}{*}{Radial basis function} & Multiquadric & \multicolumn{2}{c}{$\sqrt{r^{2} + c^{2}}$} & \multirow{2}{*}{$c$ not too large}\\
    & Inverse multiquadric & \multicolumn{2}{c}{$1 / \sqrt{r^{2} + c^{2}}$}\\
    \hline
   \end{tabular}
  \end{center}
 \end{table}
 The latter are well-known within the radial basis function (RBF) community and have been used to successfully model smooth surfaces \cite{carr:2001:proc-28th-annu-conf-computer-graphics-interact-technique,hardy:1971:j-geophys-res}. Also of note is the 2D biharmonic Green's function, the so-called thin plate spline
 \begin{equation}
  \phi_{\TPS} \left( r \right) = r^{2} \log r,
  \label{eqn:thin-plate-spline}
 \end{equation}
 which minimizes a physical bending energy \cite{duchon:1977:lect-notes-math}. For an overview of RBFs, see \cite{buhmann:2003:cambridge,powell:1987:algorithm-approx}.

 {\em Remark}. Although we focus in this paper on dense matrices, many sparse matrices, e.g., those resulting from local finite difference-type discretizations, are also of HBS form.

 Previous work on HBS matrices exploited their structure to build fast direct solvers for the square $M = N$ case \cite{gillman:2012:front-math-china,ho:2012:siam-j-sci-comput,martinsson:2005:j-comput-phys} (similar methods are available for other structured formats). Here, we extend the approach of \cite{ho:2012:siam-j-sci-comput} to the rectangular $M \geq N$ case. Our algorithm relies on the multilevel compression and sparsity-revealing embedding of \cite{ho:2012:siam-j-sci-comput}, and recasts the (unconstrained) dense least squares problem (\ref{eqn:overdetermined}) as a larger, equality-constrained, sparse one. This is solved via a sparse QR factorization coupled with iterative weighted least squares methods. For the former, we use the SuiteSparseQR package \cite{davis:2011:acm-trans-math-softw} by Davis, while for the latter, we employ the iteration of Barlow and Vemulapati \cite{barlow:1992:siam-j-numer-anal}, which has been shown to require at most two steps for problems that are not too ill-conditioned. Thus, our solver is a semi-direct method where the iteration often converges extremely quickly; in such cases, it retains all of the advantages of traditional direct solvers.

 It is useful to divide our algorithm into two phases: a direct precomputation phase, comprising matrix compression and factorization, followed by an iterative solution phase using the precomputed QR factors. Clearly, for a given matrix, only the solution phase must be executed for each additional right-hand side. Table \ref{tab:complexity} lists asymptotic complexities for both phases when applied to the operators in Table \ref{tab:examples} on data embedded in a $d$-dimensional domain for $d = 1$, $2$, or $3$.
 \begin{table}
  \caption{Asymptotic complexities for the least squares solver when applied to the operators in Table \ref{tab:examples} on data embedded in a $d$-dimensional domain: $M$ and $N$, matrix dimensions; $M \geq N$.}
  \label{tab:complexity}
  \begin{center}
   \footnotesize
   \begin{tabular}{cll}
    \hline
    $d$ & Precomputation & Solution\\
    \hline
    $1$ & $\bigO (M + N)$ & $\bigO (M + N)$\\
    $2$ & $\bigO (M + N^{3/2})$ & $\bigO (M + N \log N)$\\
    $3$ & $\bigO (M + N^{2})$ & $\bigO (M + N^{4/3})$\\
    \hline
   \end{tabular}
  \end{center}
 \end{table}
 Although the estimates generally worsen as $d$ increases, the solver achieves optimal $\bigO (M + N)$ complexity for both phases in {\em any} dimension in the special case that the source (column) and target (row) data are separated (i.e., the domain and range of the continuous operator are disjoint). This may have applications, for example, in partial charge fitting in computational chemistry \cite{bayly:1993:j-phys-chem,francl:1996:j-comput-chem}.

 {\em Remark}. The increase in cost with $d$ is due to the singular nature of the kernels in Table \ref{tab:examples} at the origin, which leads to growth of the off-diagonal block ranks defining the HBS form (section \ref{sec:complexity}). If the data are separated or if the kernel itself is smooth, then this rank growth does not occur. In this paper, we will not specifically address this latter setting, viewing it instead as a special case of the 1D problem.

 Our methods can also generalize to underdetermined systems ($M < N$) when seeking the minimum-norm solution in $L^{2}$, i.e., the equality-constrained least squares problem
 \begin{equation}
  \min_{Ax = b} \| x \|,
  \label{eqn:underdetermined}
 \end{equation}
 provided that the solution, which is also given by (\ref{eqn:solution}), is not too ill-conditioned with respect to $A$.

 Fast direct least squares algorithms have been developed in other structured matrix contexts as well, in particular within the $\mathscr{H}$- and HSS matrix frameworks using various structured orthogonal transformation schemes \cite{benner:2010:computing,chandrasekaran:2005:siam-j-matrix-anal-appl,dewilde:2006:oper-theor-adv-appl}. Our approach, however, is quite different and explicitly leverages the sparse representation of HBS matrices and the associated sparse matrix technology (e.g., the state-of-the-art software package SuiteSparseQR). This has the possible advantage of producing an algorithm that is easier to implement, extend, and optimize. For example, although we consider here only the standard Moore-Penrose systems (\ref{eqn:overdetermined}) and (\ref{eqn:underdetermined}), it is immediate that our techniques can be applied to general equality-constrained least squares problems with any combination of the system and constraint matrices being HBS (but with a possible increase in cost). For related work on other structured matrices including those of Toeplitz, Hankel, and Cauchy type, see, for instance, \cite{gu:1998:siam-j-matrix-anal-appl,kailath:1999:siam,van-barel:2003:linear-algebra-appl,xia:2012:siam-j-matrix-anal-appl} and references therein.

 The remainder of this paper is organized as follows. In the next two sections, we collect and review certain mathematical preliminaries on HBS matrices (section \ref{sec:hierarchical}) and equality-constrained least squares problems (section \ref{sec:equality-least-squares}). In section \ref{sec:algorithm}, we describe our fast semi-direct algorithm for both over- and underdetermined systems. Complexity estimates are given in section \ref{sec:complexity}, while section \ref{sec:updating} discusses efficient updating and downdating in the context of our solver. Numerical results for a variety of radial kernels are reported in section \ref{sec:results}. Finally, in section \ref{sec:conclusion}, we summarize our findings and end with some generalizations and concluding remarks.

 \section{HBS matrices}
 \label{sec:hierarchical}
 In this section, we define the HBS matrix property and discuss algorithms to compress such matrices and to sparsify linear systems governed by them. We will mainly follow the treatment of \cite{ho:2012:siam-j-sci-comput}, extended to rectangular matrices in the natural way.

 Let $A \in \mathbb{C}^{M \times N}$ be a matrix viewed with $p \times p$ blocks, with the $i$th row and column blocks having dimensions $m_{i}, n_{i} > 0$, respectively, for $i = 1, \dots, p$.

 \begin{definition}[block separable matrix \cite{gillman:2012:front-math-china}]
 The matrix $A$ is {\em block separable} if each off-diagonal submatrix $A_{ij}$ can be decomposed as the product of three low-rank matrices:
  \begin{equation}
   A_{ij} = L_{i} S_{ij} R_{j} \quad \mbox{for $i \neq j$},
   \label{eqn:block-separable}
  \end{equation}
  where $L_{i} \in \mathbb{C}^{m_{i} \times k^{\row}_{i}}$, $S_{ij} \in \mathbb{C}^{k^{\row}_{i} \times k^{\col}_{j}}$, and $R_{j} \in \mathbb{C}^{k^{\col}_{j} \times n_{j}}$, with (ideally) $k^{\row}_{i} \ll m_{i}$ and $k^{\col}_{j} \ll n_{j}$, for $i, j = 1, \dots, p$.
 \end{definition}

 Clearly, the block separability condition (\ref{eqn:block-separable}) is equivalent to requiring that the off-diagonal block rows and columns have low rank (Fig.\ \ref{fig:rank-structure}).
 \begin{figure}
  \centering
  \includegraphics{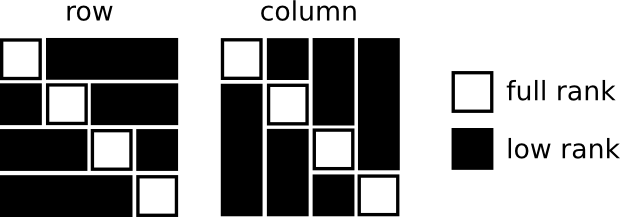}
  \caption{A block separable matrix has low-rank off-diagonal block rows and columns (black); its diagonal blocks (white) can in general be full-rank.}
  \label{fig:rank-structure}
 \end{figure}
 Observe that if $A$ is block separable, then it can be written as
 \begin{equation}
  A = D + LSR,
  \label{eqn:one-level-representation}
 \end{equation}
 where $D = \diag (A_{ii})$, $L = \diag (L_{i})$, $R = \diag (R_{i})$, and $S = (S_{ij})$ is dense with $S_{ii} = 0$.

 Let us now define a tree structure on the row and column indices $I = \{ 1, \dots, M \}$ and $J = \{ 1, \dots, N \}$, respectively, as follows. Associate with the root of the tree the entire index sets $I$ and $J$. If a given subdivision criterion is satisfied (e.g., based on the sizes $|I|$ and $|J|$), partition the root node into a set of children, each associated with a subset of $I$ and $J$ such that they together span the whole sets. Repeat this process for each new node to be subdivided, partitioning its row and column indices among its children. In other words, if we label each tree node with an integer $i$ and denote its row and column index sets by $I_{i}$ and $J_{i}$, respectively, then
 $$
  I_{i} = \bigcup_{j \in \child (i)} I_{j}, \quad J_{i} = \bigcup_{j \in \child (i)} J_{j},
 $$
 where $\child (i)$ gives the set of node indices belonging to the children of node $i$. Furthermore, we also label the levels of the tree, starting with level $0$ for the root at the coarsest level to level $\lambda$ for the leaves at the finest level; see Fig.\ \ref{fig:tree} for an example.
 \begin{figure}
  \centering
  \includegraphics{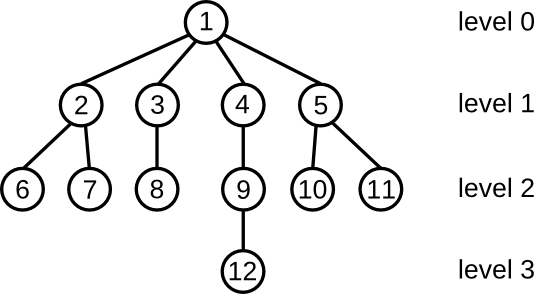}
  \caption{An example of a tree on the row and column index sets with depth $\lambda = 3$. The root (node $1$) contains all indices, which are hierarchically partitioned among its children.}
  \label{fig:tree}
 \end{figure}

 {\em Remark}. Although we require that the number of row and column partitions be the same, we do not impose that $|I_{i}| = |J_{i}|$ for any node $i$. Indeed, it is possible for one of these sets to be empty.

 Evidently, the tree defines a hierarchy among row and column index sets, each level of which specifies a block partition of the matrix $A$.

 \begin{definition}[HBS matrix \cite{gillman:2012:front-math-china}]
  The matrix $A$ is {\em HBS} if it is block separable at each level of the tree hierarchy.
 \end{definition}

 HBS matrices arise in many applications, for example, when discretizing the kernels in Table \ref{tab:examples} (up to a specified numerical precision), with row and column indices partitioned according to an octree-type ordering on the corresponding data \cite{greengard:1987:j-comput-phys,greengard:1997:acta-numer,hackbusch:2002:computing,ho:2012:siam-j-sci-comput}, which recursively groups together points that are geometrically collocated \cite{samet:1984:acm-comput-surv}.

 \subsection{Multilevel matrix compression}
 \label{sec:matrix-compression}
 We now review algorithms \cite{gillman:2012:front-math-china,ho:2012:siam-j-sci-comput,martinsson:2005:j-comput-phys} for computing the low-rank matrices in (\ref{eqn:block-separable}) characterizing the HBS form. Our primary tool for this task is the interpolative decompositon (ID) \cite{cheng:2005:siam-j-sci-comput}.

 \begin{definition}[ID]
  An {\em ID} of a matrix $A \in \mathbb{C}^{m \times n}$ with rank $k$ is a factorization $A = BP$, where $B \in \mathbb{C}^{m \times k}$ consists of a subset of the columns of $A$ and $P \in \mathbb{C}^{k \times n}$ contains the $k \times k$ identity matrix. We call $B$ and $P$ the {\em skeleton} and {\em interpolation matrices}, respectively.
 \end{definition}

 As stated, the ID clearly compresses the column space of $A$, but we can just as well compress the row space by applying the ID to $A^{\trans}$. Efficient algorithms for adaptively computing an ID to any specified precision are available \cite{cheng:2005:siam-j-sci-comput,liberty:2007:proc-natl-acad-sci-usa,woolfe:2008:appl-comput-harmon-anal}, i.e., the required rank $k$ is an {\em output} of the ID.

 \begin{definition}[row and column skeletons]
  The row indices corresponding to the retained rows in the ID are called the {\em row} or {\em incoming skeletons}; the column indices corresponding to the retained columns are called the {\em column} or {\em outgoing skeletons}.
 \end{definition}

 A multilevel algorithm for the compression of HBS matrices then follows. For simplicity, we describe the procedure for matrices with a uniform tree depth (i.e., all leaves are at level $\lambda$), with the understanding that it extends easily to the adaptive case. The following scheme \cite{gillman:2012:front-math-china,ho:2012:siam-j-sci-comput,martinsson:2005:j-comput-phys} is known as recursive skeletonization (Fig.\ \ref{fig:compression}):
 \begin{figure}
  \centering
  \includegraphics{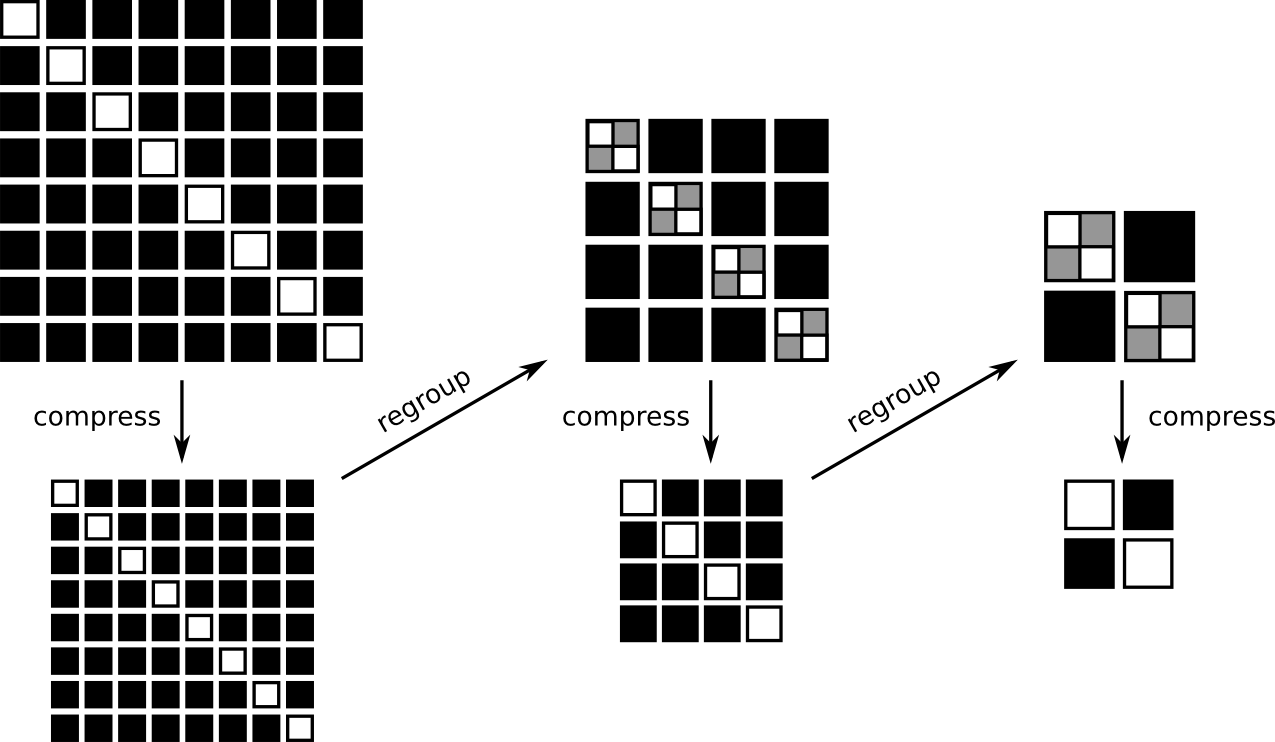}
  \caption{Schematic of recursive skeletonization, comprising alternating steps of compression and regrouping via ascension of the index tree. The diagonal blocks (white and gray) are extracted at each level; they are shown here only to illustrate the regrouping process.}
  \label{fig:compression}
 \end{figure}
 \begin{remunerate}
  \item
   Starting at the leaves of the tree, extract the diagonal blocks and compress the off-diagonal block rows and columns using the ID to a specified precision $\epsilon > 0$ as follows. For each block $i = 1, \dots, p$, compress the row space of the $i$th off-diagonal block row and call $L_{i}$ the corresponding row interpolation matrix. Similarly, for each block $j$, compress the column space of the $j$th off-diagonal block column and call $R_{j}$ the corresponding column interpolation matrix. Let $S$ be the ``skeleton'' submatrix of $A$, with each off-diagional block $S_{ij}$ for $i \neq j$ defined by the row and column skeletons associated with $L_{i}$ and $R_{j}$, respectively.
  \item
   Since the off-diagonal blocks $S_{ij}$ are submatrices of the corresponding $A_{ij}$, the compressed matrix $S$ is HBS and so can itself be compressed in the same way. Thus, move up one level in the tree, regroup the matrix blocks accordingly, and repeat.
 \end{remunerate}

 The result is a telescoping matrix representation of the form
 \begin{equation}
  A \approx D^{(\lambda)} + L^{(\lambda)} \left[ \cdots D^{(2)} + L^{(2)} \left( D^{(1)} + L^{(1)} D^{(0)} R^{(1)} \right) R^{(2)} \cdots \right] R^{(\lambda)},
  \label{eqn:multilevel-representation}
 \end{equation}
 cf.\ (\ref{eqn:one-level-representation}), where the superscript indexes the tree level $l = 0, \dots, \lambda$, that is accurate to relative precision approximately $\epsilon$. The algorithm is automatically adaptive in the sense that the compression is more efficient if lower precision is required \cite{gillman:2012:front-math-china,ho:2012:siam-j-sci-comput,martinsson:2005:j-comput-phys}.

 {\em Remark}. For the kernels in Table \ref{tab:examples}, which obey some form of Green's theorem (at least approximately), it is possible to substantially accelerate the preceding algorithm by using a ``proxy'' surface to capture all far-field interactions (see also section \ref{sec:complexity}). The key idea is that any such interaction can be represented in terms of some equivalent density on an appropriate local bounding surface, which can be chosen so that it requires only a constant number of points to discretize, irrespective of the actual number of points in the far field or their detailed structure. This observation hence replaces each global compression step with an entirely local one; see \cite{cheng:2005:siam-j-sci-comput,gillman:2012:front-math-china,greengard:2009:acta-numer,ho:2012:siam-j-sci-comput,martinsson:2005:j-comput-phys,ying:2004:j-comput-phys} for details.

 \subsection{Structured sparse embedding}
 \label{sec:sparse-embedding}
 For $M = N$, the decomposition (\ref{eqn:multilevel-representation}) enables a highly structured sparse representation \cite{chandrasekaran:2006a:siam-j-matrix-anal-appl,ho:2012:siam-j-sci-comput} of the linear system $Ax = b$ as
 \begin{equation}
  \left[
  \begin{array}{cccccccc}
   D^{(\lambda)} & L^{(\lambda)}\\
   R^{(\lambda)} & & -I\\
   & -I & D^{(\lambda - 1)} & \ddots\\
   & & \ddots & \ddots & -I\\
   & & & -I & D^{(1)} & L^{(1)}\\
   & & & & R^{(1)} & & -I\\
   & & & & & -I & D^{(0)}
  \end{array} \right] \left[
  \begin{array}{c}
   x^{(\lambda)}\\
   y^{(\lambda)}\\
   \vdots\\
   \vdots\\
   x^{(1)}\\
   y^{(1)}\\
   x^{(0)}
  \end{array} \right] = \left[
  \begin{array}{c}
   b\\
   0\\
   \vdots\\
   \vdots\\
   0\\
   0\\
   0
  \end{array} \right]
  \label{eqn:sparse-embedding}
 \end{equation}
 under the identifications
 \begin{subequations}
  \label{eqn:auxiliary-variables}
  \begin{eqnarray}
   x^{(\lambda)} = x, &\qquad x^{(l)} = R^{(l + 1)} x^{(l + 1)} \quad& \mbox{for $l = \lambda - 1, \dots, 0$},\\
   y^{(1)} = D^{(0)} x^{(0)}, &\qquad y^{(l + 1)} = D^{(l)} x^{(l)} + L^{(l)} y^{(l)} \quad& \mbox{for $l = 1, \dots, \lambda - 1$}.
  \end{eqnarray}
 \end{subequations}
 This expanded embedding clearly exposes the sparsity of HBS matrix equations and permits the immediate application of existing fast sparse solvers (such as UMFPACK \cite{davis:2004:acm-trans-math-softw} as in \cite{ho:2012:siam-j-sci-comput}).

 If $M \geq N$, however, then we have to deal with the overdetermined problem (\ref{eqn:overdetermined}), and (\ref{eqn:sparse-embedding}) must be interpreted somewhat more carefully. In particular, the identities (\ref{eqn:auxiliary-variables}) still hold, so only the first block row of (\ref{eqn:sparse-embedding}) is to be solved in the least squares sense. Thus, denoting the first block row of the sparse matrix in (\ref{eqn:sparse-embedding}) by $\mathbf{E}$ and the remainder (i.e., its last $2 \lambda$ block rows) by $\mathbf{C}$, and defining $\mathbf{x} = (x^{(\lambda)}, y^{(\lambda)}, \dots, x^{(1)}, y^{(1)}, x^{(0)})^{\trans}$, the analogue of (\ref{eqn:sparse-embedding}) for (\ref{eqn:overdetermined}) is the {\em equality-constrained} least squares problem
 \begin{equation}
  \min_{\mathbf{C} \mathbf{x} = \mathbf{0}} \| \mathbf{E} \mathbf{x} - b \|,
  \label{eqn:overdetermined-embedding}
 \end{equation}
 where both $\mathbf{E}$ and $\mathbf{C}$ are sparse. It is easy to see that $\mathbf{C}$ has full row rank.

 Similarly, if $M < N$ and we seek to solve the underdetermined system (\ref{eqn:underdetermined}), then the corresponding problem is
 \begin{equation}
  \min_{\mathbf{M} \mathbf{x} = \mathbf{b}_{1}} \left\| \mathbf{I}_{1} \mathbf{x} \right\|,
  \label{eqn:underdetermined-embedding}
 \end{equation}
 where
 \begin{equation}
  \mathbf{M} = \left[
  \begin{array}{c}
   \mathbf{E}\\
   \mathbf{C}
  \end{array} \right], \quad \mathbf{b}_{1} = \left[
  \begin{array}{c}
   b\\
   0\\
   \vdots\\
   0
  \end{array} \right], \quad \mathbf{I}_{1} = \left[
  \begin{array}{cccc}
   I & 0 & \cdots & 0
  \end{array} \right],
  \label{eqn:sparse-matrix}
 \end{equation}
 i.e., $\mathbf{M}$ is the entire sparse matrix on the left-hand side of (\ref{eqn:sparse-embedding}), which also has full row rank; $\mathbf{b}_{1}$ is the right-hand side of (\ref{eqn:sparse-embedding}); and $\mathbf{I}_{1}$ is an operator that picks out the first block row of the vector on which it acts.

 \section{Equality-constrained least squares}
 \label{sec:equality-least-squares}
 We now turn to the solution of linear least squares problems with linear equality constraints, with special attention to the case that both governing matrices are sparse as in (\ref{eqn:overdetermined-embedding}) and (\ref{eqn:underdetermined-embedding}). For consistency with the linear algebra literature, we adopt the notation of Barlow \etal \cite{barlow:1988:siam-j-numer-anal,barlow:1988:siam-j-sci-stat-comput,barlow:1992:siam-j-numer-anal}, which unfortunately conflicts somewhat with our previous definitions; the following notation is thus meant to pertain only to this section.

 Hence, consider the problem
 \begin{equation}
  \min_{Cx = g} \| Ex - f \|,
  \label{eqn:equality-least-squares}
 \end{equation}
 where $E \in \mathbb{C}^{m \times n}$ and $C \in \mathbb{C}^{p \times n}$, with
 $$
  \rank (C) = p, \quad \rank \left( \left[
  \begin{array}{c}
   E\\
   C
  \end{array} \right] \right) = n
 $$
 so that the solution is unique. Classical reduction schemes for solving (\ref{eqn:equality-least-squares}), such as the direct elimination and nullspace methods, require matrix products that can destroy the sparsity of the resulting reduced, unconstrained systems \cite{bjorck:1996:siam,lawson:1974:prentice-hall}.

 \subsection{Weighted least squares}
 An attractive alternative when both $E$ and $C$ are sparse is the method of weighting, which recasts (\ref{eqn:equality-least-squares}) in the unconstrained form
 \begin{equation}
  \min_{x} \| A(\tau) x - b(\tau) \|,
  \label{eqn:weighted-least-squares}
 \end{equation}
 where
 $$
  A(\tau) = \left[
  \begin{array}{c}
   E\\
   \tau C
  \end{array} \right], \quad b(\tau) = \left[
  \begin{array}{c}
   f\\
   \tau g
  \end{array} \right]
 $$
 for $\tau$ a suitably large weight. Clearly, as $\tau \to \infty$, the solution of (\ref{eqn:weighted-least-squares}) approaches that of (\ref{eqn:equality-least-squares}). The advantage, of course, is that (\ref{eqn:weighted-least-squares}) can be solved using standard sparse techniques; this point of view is elaborated in \cite{barlow:1988:siam-j-sci-stat-comput,van-loan:1985:siam-j-numer-anal}.

 However, the choice of an appropriate weight can be a delicate matter: if $\tau$ is too small, then (\ref{eqn:weighted-least-squares}) approximates (\ref{eqn:equality-least-squares}) poorly, while if $\tau$ is too large, then (\ref{eqn:weighted-least-squares}) can be ill-conditioned. An intuitive approach is to start with a small weight, then carry out some type of iterative refinement, effectively increasing the weight with each step. Such a scheme was first proposed by Van Loan \cite{van-loan:1985:siam-j-numer-anal}, then further studied and improved by Barlow \etal \cite{barlow:1988:siam-j-numer-anal,barlow:1992:siam-j-numer-anal}; we summarize their results in the next section.

 \subsection{Iterative reweighting by deferred correction}
 \label{sec:deferred-correction}
 In \cite{barlow:1992:siam-j-numer-anal}, Barlow and Vemulapati presented the following deferred correction procedure for the solution of the equality-constrained least squares problem (\ref{eqn:equality-least-squares}) via the successive solution of the weighted problem (\ref{eqn:weighted-least-squares}) with a {\em fixed} weight $\tau$:
 \begin{remunerate}
  \item
   Find
   $$
    x^{(0)} = \argmin_{x} \| A(\tau) x - b(\tau) \|
   $$
   and set
   $$
    r^{(0)} = f - E x^{(0)}, \quad w^{(0)} = g - C x^{(0)}, \quad \lambda^{(0)} = \tau^{2} w^{(0)}.
   $$
  \item
   For $k = 0, 1, 2, \dots$ until convergence, find
   $$
    \Delta x^{(k)} = \argmin_{x} \| A(\tau) x - b^{(k)} (\tau) \|, \quad b^{(k)} (\tau) = \left[
    \begin{array}{c}
     r^{(k)}\\
     \tau w^{(k)} + \tau^{-1} \lambda^{(k)}
    \end{array} \right]
   $$
   and update
   \begin{eqnarray*}
    x^{(k + 1)} &=& x^{(k)} + \Delta x^{(k)},\\
    r^{(k + 1)} &=& r^{(k)} - E \Delta x^{(k)},\\
    w^{(k + 1)} &=& w^{(k)} - C \Delta x^{(k)},\\
    \lambda^{(k + 1)} &=& \lambda^{(k)} + \tau^{2} w^{(k + 1)}.
   \end{eqnarray*}
   Terminate when the constraint residual $\| w^{(k + 1)} \|$ is small.
 \end{remunerate}

 Since $\tau$ is fixed, a single precomputed QR factorization of $A(\tau)$ can be used for all iterations. This algorithm is a slight modification of that employed by Van Loan \cite{van-loan:1985:siam-j-numer-anal} and has been shown to converge to the correct solution for $\tau$ appropriately chosen, provided that (\ref{eqn:equality-least-squares}) is not too ill-conditioned \cite{barlow:1988:siam-j-numer-anal,barlow:1992:siam-j-numer-anal,van-loan:1985:siam-j-numer-anal}. In particular, if implemented in double precision, then for $\tau = \epsilon_{\mach}^{-1/3} \sim 1.7 \times 10^{5}$, where $\epsilon_{\mach}$ is the machine epsilon, Barlow and Vemulapati \cite{barlow:1992:siam-j-numer-anal} showed that their algorithm requires no more than two iterations. Thus, for a broad class of problems for which it is reasonable to expect an accurate answer, the above scheme often converges extremely rapidly (and can, in some sense, even be considered a {\em direct} method, which can be made explicit by running the iteration for exactly two steps).

 {\em Remark}. Although Barlow and Vemulapati \cite{barlow:1992:siam-j-numer-anal} considered (\ref{eqn:equality-least-squares}) only over the reals, there is no inherent difficulty in extending their solution procedure to the complex case.

 {\em Remark}. It was recently pointed out to us by Eduardo Corona (personal communication, Aug.\ 2013) that deferred correction can be applied to ill-conditioned systems as well, provided that $\tau$ is changed appropriately. The relevant analysis can be found in \cite[Corollary 3.1]{barlow:1988:siam-j-numer-anal}, which suggests choosing $\tau = \epsilon_{\mach}^{-1/3} [\kappa (M)]^{1/3}$, where $\kappa (M) = \| M \| \| M^{+} \|$ is the condition number of the ``stacked'' matrix
 \begin{equation}
  M = \left[
  \begin{array}{c}
   E\\
   C
  \end{array} \right].
  \label{eqn:stacked-matrix}
 \end{equation}
 Of course, setting $\tau$ now requires an estimate of $\kappa (M)$, which may not always be available; for this reason, we have elected simply to present our algorithm with $\tau = \epsilon_{\mach}^{-1/3}$.

 \section{Algorithm}
 \label{sec:algorithm}
 We are now in a position to describe our fast semi-direct method for solving the over- and underdetermined systems (\ref{eqn:overdetermined}) and (\ref{eqn:underdetermined}), respectively, when $A$ is HBS. Let $\epsilon > 0$ be a specified numerical precision and set $\tau = \epsilon_{\mach}^{-1/3} \sim 1.7 \times 10^{5}$; we assume that all calculations are performed in double precision.

 \subsection{Overdetermined systems}
 Let $A \in \mathbb{C}^{M \times N}$ be HBS with $M \geq N$. Our algorithm proceeds in two phases. First, for the precomputation phase:
 \begin{remunerate}
  \item
   Compress $A$ to precision $\epsilon$ using recursive skeletonization \cite{ho:2012:siam-j-sci-comput}.
  \item
   Compute a sparse QR factorization of the weighted sparse matrix
   $$
    \mathbf{A} = \left[
    \begin{array}{c}
     \mathbf{E}\\
     \tau \mathbf{C}
    \end{array} \right],
   $$
   where $\mathbf{E}$ and $\mathbf{C}$ are as defined in section \ref{sec:sparse-embedding}.
 \end{remunerate}

 This is followed by the solution phase, which, for a given right-hand side $b$, produces an approximate solution (\ref{eqn:solution}) by using the precomputed QR factors to solve the equality-constrained least squares embedding (\ref{eqn:overdetermined-embedding}) via deferred correction \cite{barlow:1992:siam-j-numer-anal}. Clearly, for a fixed matrix $A$, only the solution phase must be performed for each additional right-hand side. Therefore, the cost of the precomputation phase is amortized over all such solves.

 {\em Remark}. The algorithm can also easily accommodate various modifications of the standard system (\ref{eqn:overdetermined}), e.g., the problem
 \begin{equation}
  \min_{x} \| Ax - b \|^{2} + \mu^{2} \| x \|^{2}
  \label{eqn:overdetermined-regularized}
 \end{equation}
 with Tikhonov regularization, which can be solved using
 $$
  \mathbf{A} = \left[
  \begin{array}{c}
   \mathbf{E}\\
   \mu \mathbf{I}_{1}\\
   \tau \mathbf{C}
  \end{array} \right], \quad \mathbf{b} = \left[
  \begin{array}{c}
   b\\
   0\\
   0
  \end{array} \right]
 $$
 in the weighted formulation (\ref{eqn:weighted-least-squares}).

 \subsection{Underdetermined systems}
 Now let $A \in \mathbb{C}^{M \times N}$ be HBS with $M < N$. Then (\ref{eqn:underdetermined}) can be solved using the same algorithm as above but with
 $$
  \mathbf{A} = \left[
  \begin{array}{c}
   \mathbf{I}_{1}\\
   \tau \mathbf{M}
  \end{array} \right], \quad \mathbf{b} = \left[
  \begin{array}{c}
   0\\
   \mathbf{b}_{1}
  \end{array} \right].
 $$

 \subsection{Error analysis}
 We now give an informal discussion of the accuracy of our method. Assume that $M \geq N$ and let $\kappa (A) = \| A \| \| A^{+} \| = \sigma_{1} (A) / \sigma_{N} (A)$ be the condition number of $A$, where $\sigma_{1} (A) \geq \cdots \geq \sigma_{N} (A) \geq \sigma_{N + 1} (A) = \cdots = \sigma_{M} (A) = 0$ are the singular values of $A$. We first estimate the condition number of the compressed matrix $A_{\epsilon}$ in (\ref{eqn:multilevel-representation}).

 \begin{proposition}
  Let $A_{\epsilon} = A + E$ with $\| E \| = \bigO (\epsilon \| A \|)$. If $\epsilon \kappa (A) \ll 1$, then $\kappa (A_{\epsilon}) = \bigO (\kappa (A))$.
 \end{proposition}

 \begin{proof}
  By Weyl's inequality,
  $$
   |\sigma_{i} (A_{\epsilon}) - \sigma_{i} (A)| \leq \| E \|, \quad i = 1, \dots, N,
  $$
  so $\sigma_{i} (A_{\epsilon}) = \sigma_{i} (A) + e_{i}$ with $|e_{i}| = \bigO (\epsilon \| A \|)$. Therefore,
  $$
   \kappa (A_{\epsilon}) = \frac{\sigma_{1} (A_{\epsilon})}{\sigma_{N} (A_{\epsilon})} = \frac{\sigma_{1} (A) + e_{1}}{\sigma_{N} (A) + e_{N}} = \kappa (A) \left[ \frac{1 + e_{1} / \sigma_{1} (A)}{1 + e_{N} / \sigma_{N} (A)} \right],
  $$
  where
  $$
   \frac{|e_{1}|}{\sigma_{1} (A)} = \frac{|e_{1}|}{\| A \|} = \bigO (\epsilon), \quad \frac{|e_{N}|}{\sigma_{N} (A)} = e_{N} \| A^{+} \| = \bigO (\epsilon \kappa (A)),
  $$
  so $\kappa (A_{\epsilon}) = \bigO (\kappa (A))$.
 \end{proof}

 In other words, if $A$ is not too ill-conditioned, then neither is $A_{\epsilon}$. But the convergence of deferred correction depends on the conditioning of the stacked matrix (\ref{eqn:stacked-matrix}), i.e., the sparse embedding $\mathbf{M}$ of $A_{\epsilon}$ in (\ref{eqn:sparse-matrix}). Although we have not explicitly studied the spectral properties of $\mathbf{M}$, numerical estimates suggest that $\kappa (\mathbf{M}) = \bigO (\kappa (A_{\epsilon}))$. Some evidence for this can be seen in the square, single-level case (\ref{eqn:one-level-representation}), for which the analogue of $\mathbf{M}$ is
 $$
  \left[
  \begin{array}{ccc}
   D & L\\
   R & & -I\\
   & -I & S
  \end{array} \right] = \left[
  \begin{array}{ccc}
   I & -LS & -L\\
   & I\\
   & & I
  \end{array} \right] \left[
  \begin{array}{ccc}
   A_{\epsilon}\\
   & & -I\\
   & -I & S
  \end{array} \right] \left[
  \begin{array}{ccc}
   I\\
   -SR & I\\
   -R & & I
  \end{array} \right],
 $$
 with inverse
 $$
  \left[
  \begin{array}{ccc}
   D & L\\
   R & & -I\\
   & -I & S
  \end{array} \right]^{-1} = \left[
  \begin{array}{ccc}
   I & LS & L\\
   & I\\
   & & I
  \end{array} \right] \left[
  \begin{array}{ccc}
   A_{\epsilon}^{-1}\\
   & S & -I\\
   & -I
  \end{array} \right] \left[
  \begin{array}{ccc}
   I\\
   SR & I\\
   R & & I
  \end{array} \right].
 $$
 Assume without loss of generality that $\| A \| = 1$. Since $S$ is a submatrix of $A$, this also implies that $\| S \| \leq 1$. Furthermore, it is typically the case that $\| L \|$ and $\| R \|$ are not too large since they come from the ID \cite{cheng:2005:siam-j-sci-comput}. It then follows that $\| \mathbf{M} \| = \bigO (\| A_{\epsilon} \|)$ and $\| \mathbf{M}^{-1} \| = \bigO (\| A_{\epsilon}^{-1} \|)$; thus, $\kappa (\mathbf{M}) = \bigO (\kappa (A_{\epsilon})) = \bigO (\kappa (A))$.

 This argument can be extended to the multilevel setting (\ref{eqn:sparse-embedding}), but only for square matrices. Still, in practice, we observed that the claim seems to hold also for rectangular matrices, so that if $\kappa (A)$ is not too large then neither is $\kappa (\mathbf{M})$ and we may expect deferred correction to succeed.

 Even if the iteration converges to the exact solution, however, there is still an error arising from the use of $A_{\epsilon}$ in place of $A$ itself. The extent of this error is governed by standard perturbation theory. The following is a restatement of Theorem 20.1 in \cite{higham:2002:siam}, originally due to Wedin, specialized to the current setting.

 \begin{theorem}[Wedin]
  Let $A \in \mathbb{C}^{M \times N}$ be full-rank with $M \geq N$, and let
  $$
   x = \argmin_{y} \| Ay - b \|, \quad x_{\epsilon} = \argmin_{y} \left\| A_{\epsilon} y - b \right\|
  $$
  be the solutions of the corresponding overdetermined systems, with residuals
  $$
   r = b - Ax, \quad r_{\epsilon} = b - A_{\epsilon} x_{\epsilon},
  $$
  respectively. If $\epsilon \kappa (A) < 1$, then
  $$
   \frac{\left\| x - x_{\epsilon} \right\|}{\| x \|} \leq \frac{\epsilon \kappa (A)}{1 - \epsilon \kappa (A)} \left[ 1 + \kappa (A) \frac{\| r \|}{\| A \| \| x \|} \right], \quad \frac{\left\| r - r_{\epsilon} \right\|}{\| b \|} \leq 2 \epsilon \kappa (A).
  $$
 \end{theorem}

 A somewhat simpler bound holds for underdetermined systems \cite[Theorem 21.1]{higham:2002:siam}.
 \begin{theorem}[Demmel and Higham]
 Let $A \in \mathbb{C}^{M \times N}$ be full-rank with $M < N$, and let $x$ and $x_{\epsilon}$ be the minimum-norm solutions to the underdetermined systems $Ax = b$ and $A_{\epsilon} x_{\epsilon} = b$, respectively, for $b \neq 0$. If $\| A^{+} (A - A_{\epsilon}) \| < 1$, then
  $$
   \frac{\left\| x - x_{\epsilon} \right\|}{\| x \|} \leq 2 \epsilon \kappa (A).
  $$
 \end{theorem}

 A more rigorous analysis is not yet available, but we note that our numerical results (section \ref{sec:results}) indicate that the algorithm is accurate and stable.

 \section{Complexity analysis}
 \label{sec:complexity}
 In this section, we analyze the complexity of our solver for a representative example: the HBS matrix $A$ defined by a kernel from Table \ref{tab:examples}, acting on source and target data distributed uniformly over the same $d$-dimensional domain (but at different densities). We follow the approach of \cite{ho:2012:siam-j-sci-comput}.

 Sort both sets of data together in one hyperoctree (the multidimensional generalization of an octree, cf.\ \cite{samet:1984:acm-comput-surv}) as outlined in section \ref{sec:hierarchical}, subdividing each node until it contains no more than a set number of combined row and column indices, i.e., $|I_{i}| + |J_{i}| = \bigO (1)$ for each leaf $i$. For each tree level $l = 0, \dots, \lambda$, let $p_{l}$ denote the number of matrix blocks; $m_{l}$ and $n_{l}$, the row and column block sizes, respectively, in the compressed representation (\ref{eqn:multilevel-representation}), assumed equal across all blocks for simplicity; and $k_{l}$, the skeleton block size, which is clearly of the same order for both rows and columns as it depends only on $\min \{ \bigO (m_{l}), \bigO (n_{l}) \}$ (this can be made precise by explicitly considering proxy compression, which produces interaction matrices of size $\bigO (m_{l}) \times \bigO (n_{l})$). Note that $m_{l}$ and $n_{l}$ are {\em not} the row and column block sizes in the tree; they are the result of hierarchically ``pulling up'' skeletons during the compression process; see (iv) below. Moreover, since $M \neq N$ in general, we can define an additional level parameter $\lambda' \leq \lambda$ corresponding to the depth of the tree constructed via the same process on only the smaller of the source or target data, e.g., on only the source data if $M \geq N$. For the remainder of this discussion, we assume that $M \geq N$. Analogous results can be recovered for $M < N$ simply by switching the roles of $M$ and $N$ in what follows. We start with some useful observations:
 \begin{romannum}
  \item
   By construction, $p_{\lambda} (m_{\lambda} + n_{\lambda}) = M + N \sim M$, where $m_{\lambda}, n_{\lambda} = \bigO (1)$, so $p_{\lambda} \sim M$. Similarly, $p_{\lambda'} \sim N$.
  \item
   Each subdivision increases the number of blocks by a factor of roughly $2^{d}$, so $p_{l + 1} \sim 2^{d} p_{l}$. In particular, $p_{0} = 1$, so $\lambda \sim (1/d) \log M$ and $\lambda' \sim (1/d) \log N$.
  \item
   For $l = \lambda' + 1, \dots, \lambda$, $k_{l} = \bigO (1)$ since $\min \{ \bigO (m_{l}), \bigO (n_{l}) \} = \bigO (1)$, while for $l = 0, \dots, \lambda'$, it can be shown \cite{ho:2012:siam-j-sci-comput} that
   \begin{equation}
    k_{l} \sim \left\{
    \begin{array}{ll}
     l & \mbox{if $d = 1$},\\
     2^{(d - 1)l} & \mbox{if $d > 1$}.
    \end{array} \right.
    \label{eqn:rank-growth}
   \end{equation}
  \item
   The total number of row and column indices at level $l < \lambda$ is equal to the total number of skeletons at level $l + 1$, i.e., $p_{l} m_{l} = p_{l} n_{l} = p_{l + 1} k_{l + 1}$, so $m_{l}, n_{l} \sim k_{l + 1}$.
 \end{romannum}

 \subsection{Matrix compression}
 From \cite{cheng:2005:siam-j-sci-comput,liberty:2007:proc-natl-acad-sci-usa,woolfe:2008:appl-comput-harmon-anal}, the cost of computing a rank-$k$ ID of an $m \times n$ matrix is $\bigO (kmn)$. If proxy compression is used, then $m \sim m_{l}$ for a block at level $l$. Therefore, the total cost of matrix compression is
 $$
  T_{\cm} \sim \sum_{l = 0}^{\lambda} p_{l} k_{l} m_{l} n_{l} \sim \sum_{l = 0}^{\lambda} p_{l} k_{l}^{3}.
 $$
 We break this into two sums, one over $l = \lambda' + 1, \dots, \lambda$ and another over $l = 0, \dots, \lambda'$, with estimates
 $$
  \sum_{l = \lambda' + 1}^{\lambda} p_{l} k_{l}^{3} \sim M, \qquad \sum_{l = 0}^{\lambda'} p_{l} k_{l}^{3} \sim \left\{
  \begin{array}{ll}
   N & \mbox{if $d = 1$,}\\
   N^{3(1 - 1/d)} & \mbox{if $d > 1$,}
  \end{array} \right.
 $$
 respectively. Hence,
 \begin{equation}
  T_{\cm} \sim \left\{
  \begin{array}{ll}
   M + N & \mbox{if $d = 1$},\\
   M + N^{3(1 - 1/d)} & \mbox{if $d > 1$}.
  \end{array} \right.
  \label{eqn:complexity-cm}
 \end{equation}

 \subsection{Compressed QR factorization}
 We consider the QR decomposition of a block tridiagonal matrix with the same sparsity structure as that of $\mathbf{M}$ in (\ref{eqn:sparse-embedding}), computed using Householder reflections. This clearly encompasses the factorization of $\mathbf{A} = \mathbf{Q} \mathbf{R}$ for both over- and underdetermined systems. We begin by studying the square $M = N$ case, for which it is immediate that $\mathbf{R}$ is block upper bidiagonal (Fig.\ \ref{fig:qr-square}), so the cost of Householder triangularization is
 \begin{equation}
  T_{\qr} \sim \sum_{l = 0}^{\lambda} p_{l} m_{l} n_{l}^{2} \sim \sum_{l = 0}^{\lambda} p_{l} k_{l}^{3} \sim \left\{
  \begin{array}{ll}
   M + N & \mbox{if $d = 1$},\\
   M + N^{3(1 - 1/d)} & \mbox{if $d > 1$},
  \end{array} \right.
  \label{eqn:complexity-qr}
 \end{equation}
 following the structure of $\mathbf{R}$.
 \begin{figure}
  \centering
  \includegraphics{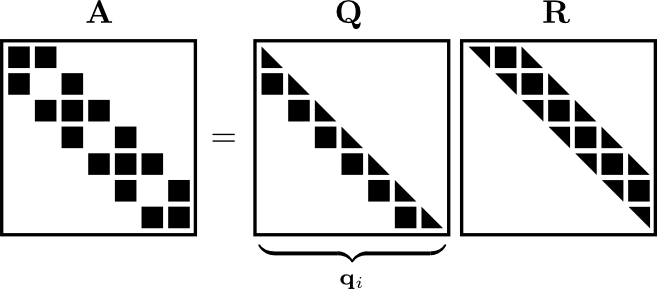}
  \caption{Sparsity structure of QR factors for the structured embedding (\ref{eqn:sparse-embedding}) with $M = N$, where the orthogonal matrix $\mathbf{Q}$ is given in terms of the elementary Householder reflectors $\mathbf{q}_{i}$.}
  \label{fig:qr-square}
 \end{figure}
 The $M > N$ and $M < N$ cases are easily analyzed by noting that only the blocks at level $\lambda$ are rectangular (in the asymptotic sense), so any nonzero propagation during triangularization is limited and both essentially reduce to the square case above (Fig. \ref{fig:qr-rectangular}).
 \begin{figure}
  \centering
  \includegraphics{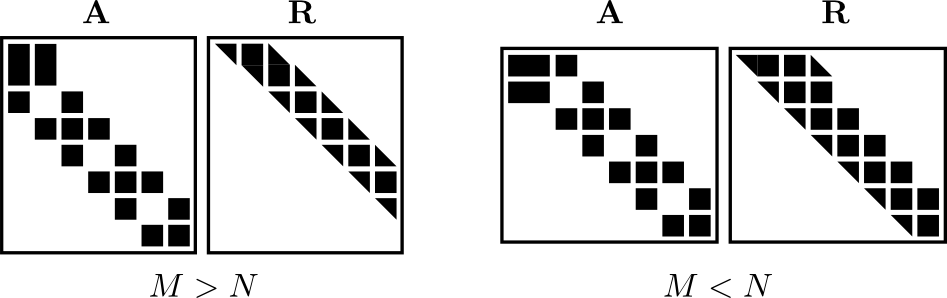}
  \caption{Sparsity structure of the triangular QR factor $\mathbf{R}$ for the structured embedding (\ref{eqn:sparse-embedding}) in the $M > N$ (left) and $M < N$ (right) cases.}
  \label{fig:qr-rectangular}
 \end{figure}

 The complexities (\ref{eqn:complexity-cm}) and (\ref{eqn:complexity-qr}) of compression and factorization, respectively, together constitute the cost of the precomputation phase.

 \subsection{Compressed pseudoinverse application}
 We now examine the cost of applying the pseudoinverse to solve the weighted least squares problem (\ref{eqn:weighted-least-squares}) using the precomputed QR factors. For this, we suppose that the solution is determined via the equation $\mathbf{R} \mathbf{x} = \mathbf{Q}^{*} \mathbf{b}$, which requires one application of $\mathbf{Q}^{*}$, assumed to be performed using elementary Householder transformations, and one backsolve with $\mathbf{R}$, whose cost is clearly on the same order as multiplying by $\mathbf{R}$. Then from the arguments above, it is evident that both operations have complexity
 $$
  T_{\sv} \sim \sum_{l = 0}^{\lambda} p_{l} m_{l} n_{l} \sim \sum_{l = 0}^{\lambda} p_{l} k_{l}^{2} \sim \left\{
  \begin{array}{ll}
   M + N & \mbox{if $d = 1$},\\
   M + N \log N & \mbox{if $d = 2$},\\
   M + N^{2(1 - 1/d)} & \mbox{if $d > 2$}.
  \end{array} \right.
 $$
 Since the total number of such problems to be solved is constant for each outer problem (\ref{eqn:overdetermined-embedding}), assuming fast convergence (section \ref{sec:deferred-correction}), this is also the complexity of the solution phase. As with classical matrix factorizations, the prefactor for $T_{\sv}$ is typically far smaller than that for $T_{\cm}$ or $T_{\qr}$.

 {\em Remark}. One can also use the seminormal equations $\mathbf{R}^{*} \mathbf{R} x = \mathbf{A}^{*} \mathbf{b}$, which do not require the orthogonal matrix $\mathbf{Q}$ \cite{bjorck:1996:siam}. However, one step of iterative refinement is necessary for stability, so the total cost is three applications of $\mathbf{A}$ or $\mathbf{A}^{*}$ (one each for the original and refinement solves, plus another to compute the residual), and four solves with $\mathbf{R}$ or $\mathbf{R}^{*}$. In practice, we found this approach to be slower than that involving $\mathbf{Q}$ by a factor of about four.

 \subsection{Some remarks}
 For all complexities above, the constants implicit in the estimates are of the form $\bigO (2^{d} \log^{\alpha} \epsilon)$ for modest $\alpha$, i.e., they are exponential in the dimension and polylogarithmic in the precision \cite{greengard:1987:j-comput-phys,greengard:1997:acta-numer}.

 In the special case that the source and target data are separated, $k_{l} = \bigO (1)$ for all $l$, so $T_{\cm}$, $T_{\qr}$, and $T_{\sv}$ all have optimal complexity $\bigO (M + N)$ in {\em any} dimension. This describes, for example, the fitting of atomic partial charges to reproduce electrostatic potential values on ``shells'' around a molecule \cite{bayly:1993:j-phys-chem,francl:1996:j-comput-chem}, the computation of equivalent densities in the kernel-independent FMM \cite{ying:2004:j-comput-phys}, and even the calculation of the ID in recursive skeletonization \cite{gillman:2012:front-math-china,ho:2012:siam-j-sci-comput,martinsson:2005:j-comput-phys}, which requires a least squares solve \cite{cheng:2005:siam-j-sci-comput}.

 \section{Updating and downdating}
 \label{sec:updating}
 We now discuss an important feature of our solver: its capacity for efficient updating and downdating in response to dynamically changing data.  Our methods are based on the augmented system approach of \cite{greengard:2009:acta-numer}, extended to the least squares setting, and exploit the ability to rapidly apply $A^{+}$ via the solution phase of our algorithm, which we hereafter take as a computational primitive. Thus, suppose that we are given some base linear system (\ref{eqn:overdetermined}), focusing for simplicity on the overdetermined case, for which we have precomputed a compressed QR factorization of $A$. We consider the addition and deletion of both rows and columns, corresponding to the modification of observations and regression variables, respectively. Furthermore, we assume that such modifications are small, in particular so that the system remains overdetermined, and accommodate each case within the framework of the general equality-constrained least squares problem
 \begin{equation}
  \min_{\mathbf{C} \mathbf{x} = \mathbf{g}} \| \mathbf{E} \mathbf{x} - \mathbf{f} \|.
  \label{eqn:equality-least-squares-augmented}
 \end{equation}

 \subsection{Adding and deleting rows}
 To add $p_{\row}$ rows to the matrix $A = (a_{ij})$, and correspondingly to the vector $b = (b_{1}, \dots, b_{M})^{\trans}$, in (\ref{eqn:overdetermined}), we simply use (\ref{eqn:equality-least-squares-augmented}) with
 \begin{equation}
  \mathbf{E} = \left[
  \begin{array}{c}
   A\phantom{_{+}}\\
   C^{\row}_{+}
  \end{array} \right], \quad \mathbf{C} = 0, \quad \mathbf{x} = x, \quad \mathbf{f} = \left[
  \begin{array}{c}
   b\phantom{_{+}}\\
   b_{+}
  \end{array} \right], \quad \mathbf{g} = 0,
  \label{eqn:row-addition}
 \end{equation}
 where $C^{\row}_{+} \in \mathbb{C}^{p_{\row} \times N}$ describes the influence of the variables $x$ on the new data $b_{+}$. To delete $q_{\row}$ rows with indices $k_{1}, \dots, k_{q_{\row}}$, we add $q_{\row}$ degrees of freedom to those rows to be deleted in order to enforce strict agreement with those observations as follows:
 $$
  \mathbf{E} = \left[
  \begin{array}{cc}
   A & B^{\row}_{-}
  \end{array} \right], \quad \mathbf{C} = \left[
  \begin{array}{cc}
   C^{\row}_{-} & I
  \end{array} \right], \quad \mathbf{x} = \left[
  \begin{array}{c}
   x\phantom{_{-}}\\
   x^{\row}_{-}
  \end{array} \right], \quad \mathbf{f} = b, \quad \mathbf{g} = d,
 $$
 where $B^{\row}_{-} = (\delta_{k_{i} j}) \in \mathbb{C}^{M \times q_{\row}}$, $C^{\row}_{-} = (a_{k_{i} j}) \in \mathbb{C}^{q_{\row} \times N}$, and $d = (b_{k_{1}}, \dots, b_{k_{q_{\row}}})^{\trans}$; here,
 $$
  \delta_{ij} = \left\{
  \begin{array}{ll}
   1 & \mbox{if $i = j$},\\
   0 & \mbox{if $i \neq j$}
  \end{array} \right.
 $$
 is the Kronecker delta. The simultaneous addition and deletion of rows can be achieved via a straightforward combination of the above:
 $$
  \mathbf{E} = \left[
  \begin{array}{cc}
   A\phantom{_{+}} & B^{\row}_{-}\\
   C^{\row}_{+}
  \end{array} \right], \quad \mathbf{C} = \left[
  \begin{array}{cc}
   C^{\row}_{-} & I
  \end{array} \right], \quad \mathbf{x} = \left[
  \begin{array}{c}
   x\phantom{_{-}}\\
   x^{\row}_{-}
  \end{array} \right], \quad \mathbf{f} = \left[
  \begin{array}{c}
   b\phantom{_{+}}\\
   b_{+}
  \end{array} \right], \quad \mathbf{g} = d.
 $$

 \subsection{Adding and deleting columns}
 To add $p_{\col}$ columns to $A$ and hence to the vector $x = (x_{1}, \dots, x_{N})^{\trans}$, we let
 $$
  \mathbf{E} = \left[
  \begin{array}{cc}
   A & B^{\col}_{+}
  \end{array} \right], \quad \mathbf{C} = 0, \quad \mathbf{x} = \left[
  \begin{array}{c}
   x\phantom{_{+}}\\
   x^{\col}_{+}
  \end{array} \right], \quad \mathbf{f} = b, \quad \mathbf{g} = 0,
 $$
 where $B^{\col}_{+} \in \mathbb{C}^{M \times p_{\col}}$ describes the influence of the new variables $x^{\col}_{+}$ on the data. To delete $q_{\col}$ columns with indices $l_{1}, \dots, l_{q_{\col}}$, we add $q_{\col}$ ``anti-variables'' annihilating their effects:
 $$
  \mathbf{E} = \left[
  \begin{array}{cc}
   A & B^{\col}_{-}
  \end{array} \right], \quad \mathbf{C} = \left[
  \begin{array}{cc}
   C^{\col}_{-} & I
  \end{array} \right], \quad \mathbf{x} = \left[
  \begin{array}{c}
   x\phantom{_{-}}\\
   x^{\col}_{-}
  \end{array} \right], \quad \mathbf{f} = b, \quad \mathbf{g} = 0,
 $$
 where $B^{\col}_{-} = (a_{i l_{j}}) \in \mathbb{C}^{M \times q_{\col}}$ and $C^{\col}_{-} = (\delta_{i l_{j}}) \in \mathbb{C}^{q_{\col} \times N}$. Finally, to add and delete columns simultaneously, we take
 $$
  \mathbf{E} = \left[
  \begin{array}{ccc}
   A & B^{\col}_{+} & B^{\col}_{-}
  \end{array} \right], \quad \mathbf{C} = \left[
  \begin{array}{ccc}
   C^{\col}_{-} & 0 & I
  \end{array} \right], \quad \mathbf{x} = \left[
  \begin{array}{c}
   x\phantom{_{+}}\\
   x^{\col}_{+}\\
   x^{\col}_{-}
  \end{array} \right], \quad \mathbf{f} = b, \quad \mathbf{g} = 0.
 $$

 \subsection{Simultaneous modification of rows and columns}
 The general case of modifying both rows and columns can be treated using (\ref{eqn:equality-least-squares-augmented}) with
 $$
  \mathbf{E} = \left[
  \begin{array}{cccc}
   A\phantom{_{+}} & B^{\col}_{+} & B^{\row}_{-} & B^{\col}_{-}\\
   C^{\row}_{+} & D_{+} & 0 & D_{-}
  \end{array} \right], \quad \mathbf{C} = \left[
  \begin{array}{cccc}
   C^{\row}_{-} & 0 & I & 0\\
   C^{\col}_{-} & 0 & 0 & I
  \end{array} \right]
 $$
 and
 $$
  \mathbf{x} = \left[
  \begin{array}{c}
   x\phantom{_{+}}\\
   x^{\col}_{+}\\
   x^{\row}_{-}\\
   x^{\col}_{-}
  \end{array} \right], \quad \mathbf{f} = \left[
  \begin{array}{c}
   b\phantom{_{+}}\\
   b_{+}
  \end{array} \right], \quad \mathbf{g} = \left[
  \begin{array}{c}
   d\\
   0
  \end{array} \right],
 $$
 where $D_{+} \in \mathbb{C}^{p_{\row} \times p_{\col}}$ describes the influence of $x^{\col}_{+}$ on $b_{+}$, $D_{-} = (c_{i l_{j}}) \in \mathbb{C}^{p_{\row} \times q_{\col}}$ for $C^{\row}_{+} = (c_{ij})$ accounts for the effect of column deletion on $b_{+}$ (alternatively, one can zero out the relevant columns in $C^{\row}_{+}$), and all other quantities are as defined previously.

 \subsection{Solution methods}
 The augmented system (\ref{eqn:equality-least-squares-augmented}) can be solved by deferred correction \cite{barlow:1992:siam-j-numer-anal}, where the matrix to be considered at each step is
 $$
  \mathbf{A} = \left[
  \begin{array}{c}
   \mathbf{E}\\
   \tau \mathbf{C}
  \end{array} \right] \equiv \left[
  \begin{array}{cc}
   A & B\\
   C & D
  \end{array} \right] \in \mathbb{C}^{(M + m) \times (N + n)}
 $$
 for $B \in \mathbb{C}^{M \times n}$, $C \in \mathbb{C}^{m \times N}$, and $D \in \mathbb{C}^{m \times n}$, where $m = p_{\row} + q_{\row} + q_{\col}$ and $n = p_{\col} + q_{\row} + q_{\col}$. If $m$ and $n$ are small, then an efficient approach is to compute $\mathbf{A}^{+}$ by using $A^{+}$ in various block pseudoinverse formulas (see, e.g., \cite{burns:1974:siam-j-appl-math}) or by invoking Greville's method \cite{campbell:1979:pitman,greville:1960:siam-rev}, which can construct $\mathbf{A}^{+}$ via a sequence of rank-one updates to $A^{+}$. Alternatively, one can appeal to iterative methods like GMRES \cite{saad:1986:siam-j-sci-stat-comput}, using $A^{+}$ as a preconditioner. In this approach, instead of solving
 $$
  \min_{\mathbf{x}} \| \mathbf{A} \mathbf{x} - \mathbf{b} \|, \quad \mathbf{b} = \left[
  \begin{array}{c}
   \mathbf{f}\\
   \tau \mathbf{g}
  \end{array} \right],
 $$
 we consider instead, say, the left preconditioned system
 $$
  \min_{\mathbf{x}} \| \mathbf{B} \mathbf{A} \mathbf{x} - \mathbf{B} \mathbf{b} \|
 $$
 for an appropriate choice of the preconditioner $\mathbf{B} \in \mathbb{C}^{(N + n) \times (M + m)}$. Hayami, Yin, and Ito \cite{hayami:2010:siam-j-matrix-anal-appl} showed that GMRES converges provided that $\range (\mathbf{A}) = \range (\mathbf{B}^{*})$ and $\range (\mathbf{A}^{*}) = \range (\mathbf{B})$. Therefore, suitable choices of $\mathbf{B}$ include, e.g.,
 $$
  \left[
  \begin{array}{cc}
   A^{+} & C^{*}\\
   B^{*} & D^{*}
  \end{array} \right], \quad \left[
  \begin{array}{cc}
   A^{+} & C^{*}\\
   B^{*} & D^{+}
  \end{array} \right]
 $$
 (the latter if $D$ is not rank-deficient), for which only one application of $A^{+}$ is required per iteration. Such methods also have the possible advantage of being more flexible and robust. If the total number of iterations is small, the cost of updating is therefore only $\bigO (T_{\sv})$ instead of $\bigO (T_{\cm} + T_{\qr})$ for computing the QR factorization anew, which is typically much larger (section \ref{sec:complexity}).

 {\em Remark}. If $M + m > N + n$, then it is more efficient to solve the left preconditioned system of dimension $N + n$. Similarly, if $M + m < N + n$, then it is more efficient to solve the right preconditioned system of dimension $M + m$.

 \section{Numerical results}
 \label{sec:results}
 In this section, we report some numerical results for our fast semi-direct solver, compared against LAPACK/ATLAS \cite{anderson:1999:siam,whaley:2001:parallel-comput} and an accelerated GMRES solver \cite{hayami:2010:siam-j-matrix-anal-appl,saad:1986:siam-j-sci-stat-comput} using an FMM-type scheme. We considered problems in both 2D and 3D. All matrices were block partitioned using quadtrees in 2D and octrees in 3D, uniformly subdivided until all leaf nodes contained no more than a fixed number of combined rows and columns (cf.\ sections \ref{sec:matrix-compression} and \ref{sec:complexity}), while adaptively pruning all empty nodes during the refinement process. The recursive skeletonization algorithm was implemented in Fortran and employed as described in \cite{ho:2012:siam-j-sci-comput}. Sparse QR factorizations were computed using SuiteSparseQR \cite{davis:2011:acm-trans-math-softw} through a \MATLAB\ R2012b (The MathWorks, Inc.: Natick, MA) interface, keeping all orthogonal matrices in compact Householder form. The deferred correction procedure \cite{barlow:1992:siam-j-numer-anal} was implemented in \MATLAB. All calculations were performed in double-precision real arithmetic on a single 3.10 GHz processor with 4 GB of RAM.

 For each case, where appropriate, we report the following data:
 \begin{itemize}
  \item
   $M$, $N$: the uncompressed row and column dimensions, respectively;
  \item
   $K_{\row}$, $K_{\col}$: the final row and column skeleton dimensions, respectively;
  \item
   $T_{\cm}$: the matrix compression time (s);
  \item
   $T_{\qr}$: the sparse QR factorization time (s);
  \item
   $T_{\sv}$: the pseudoinverse application time (s);
  \item
   $n_{\iter}$: the number of iterations required for deferred correction;
  \item
   $E$: the relative error $\left\| x - x_{\epsilon} \right\| / \| x \|$ with respect to the solution $x$ produced by LAPACK/ATLAS (if the problem is small enough) or FMM/GMRES; and
  \item
   $R$: the relative residual $\left\| A x_{\epsilon} - b \right\| / \| b \|$ with respect to the true operator.
 \end{itemize}
 Note that $E$ and $R$ are not the true relative error and residual but nevertheless provide a useful measure of accuracy.

 \subsection{Laplace's equation}
 For benchmarking purposes, we first applied our method to Laplace's equation
 $$
  \Delta u = 0 \quad \mbox{in $\Omega \in \mathbb{R}^{d}$}, \qquad u = f \quad \mbox{on $\partial \Omega$}
 $$
 in a simply connected interior domain with Dirichlet boundary conditions, which can be solved by writing the solution in the form of a double-layer potential
 $$
  u(\vec{x}) = \int_{\partial \Omega} \frac{\partial G}{\partial \nu_{\vec{y}}} (\| \vec{x} - \vec{y} \|) \: \sigma (\vec{y}) \: d \vec{y} \quad \mbox{for $\vec{x} \in \Omega$},
 $$
 where
 $$
  G(r) = \left\{
  \begin{array}{ll}
   -1 / (2 \pi) \log r & \mbox{if $d = 2$},\\
   1 / (4 \pi r) & \mbox{if $d = 3$}
  \end{array} \right.
 $$
 is the free-space Green's function, $\nu_{\vec{y}}$ is the unit outer normal at $\vec{y} \in \partial \Omega$, and $\sigma$ is an unknown surface density. Letting $\vec{x}$ approach the boundary, standard results from potential theory \cite{guenther:1988:prentice-hall} yield the second-kind Fredholm boundary integral equation
 \begin{equation}
  -\frac{1}{2} \sigma (\vec{x}) + \int_{\partial \Omega} \frac{\partial G}{\partial \nu_{\vec{y}}} (\| \vec{x} - \vec{y} \|) \: \sigma (\vec{y}) \: d \vec{y} = f(\vec{x})
  \label{eqn:integral-equation}
 \end{equation}
 for $\sigma$, assuming that $\partial \Omega$ is smooth. This is not a least squares problem, but it allows us to compare the performance of the sparse QR approach with our previous sparse LU results \cite{ho:2012:siam-j-sci-comput}. Of course, since the system (\ref{eqn:integral-equation}) is square, the corresponding sparse embedding (\ref{eqn:sparse-embedding}) can be solved without iteration.

 In 2D, we took as the problem geometry a $2$:$1$ ellipse, discretized via the trapezoidal rule, while in 3D, we used the unit sphere, discretized as a collection of flat triangles with piecewise constant densities. We also compared our algorithm against an FMM-accelerated GMRES solver driven by the open-source FMMLIB software package \cite{gimbutas:in-prep}, which is a fairly efficient implementation (but not optimized using the plane wave representations of \cite{greengard:1997:acta-numer}). Timing results for each case are shown in Fig.\ \ref{fig:solve-lap}, with detailed data for the recursive skeletonization scheme given in Tables \ref{tab:solve-lap2d} and \ref{tab:solve-lap3d}. The precision was set to $\epsilon = 10^{-9}$ in 2D and $10^{-6}$ in 3D.
 \begin{figure}
  \begin{center}
   \includegraphics{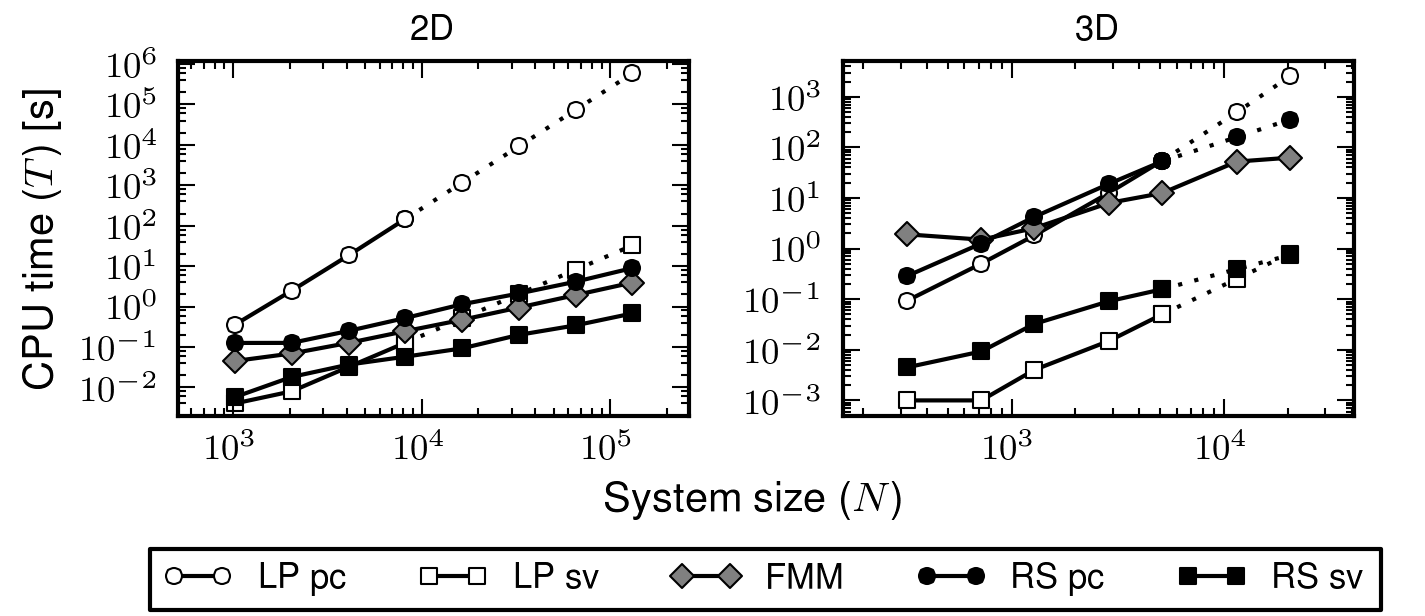}
  \end{center}
  \caption{CPU times for solving Laplace's equation in 2D and 3D using LAPACK/ATLAS (LP), FMM/GMRES (FMM), and recursive skeletonization (RS) as a function of the system size $N$. For LP and RS, the computation is split into two parts: precomputation (pc), for LP consisting of matrix formation and factorization, and for RS of matrix compression and factorization; and system solution (sv), consisting of matrix (pseudo-) inverse application via precomputed QR factors. The precision of FMM and RS was set at $\epsilon = 10^{-9}$ in 2D and $10^{-6}$ in 3D. Dotted lines indicate extrapolated data; for RS in 3D, only factorization and inversion (as executed through \MATLAB) are extrapolated.}
  \label{fig:solve-lap}
 \end{figure}

 It is evident that our method scales as predicted, with precomputation and solution complexities of $\bigO (N)$ in 2D ($d = 1$), and $\bigO (N^{3/2})$ and $\bigO (N \log N)$, respectively, in 3D ($d = 2$).
 \begin{table}
  \caption{Numerical results for solving Laplace's equation in 2D at precision $\epsilon = 10^{-9}$.}
  \label{tab:solve-lap2d}
  \begin{center}
   {\footnotesize
    \begin{tabular}{r|rrcccc}
     \hline
     \multicolumn{1}{c|}{$N$} & \multicolumn{1}{c}{$K_{\row}$} & \multicolumn{1}{c}{$K_{\col}$} & $T_{\cm}$ & $T_{\qr}$ & $T_{\sv}$ & $E$\\
     \hline
       1024 & 30 & 30 & 3.1\E$-2$ & 9.4\E$-2$ & 5.8\E$-3$ & 1.1\E$-9$\\
       2048 & 29 & 30 & 6.5\E$-2$ & 5.8\E$-2$ & 1.8\E$-2$ & 4.5\E$-9$\\
       4096 & 30 & 30 & 1.3\E$-1$ & 1.2\E$-1$ & 3.7\E$-2$ & 1.5\E$-8$\\
       8192 & 30 & 31 & 2.6\E$-1$ & 2.6\E$-1$ & 5.7\E$-2$ & 1.4\E$-8$\\
      16384 & 31 & 31 & 5.2\E$-1$ & 6.0\E$-1$ & 9.2\E$-2$ & 1.7\E$-8$\\
      32768 & 30 & 30 & 1.1\E$+0$ & 1.0\E$+0$ & 2.0\E$-1$ & 1.2\E$-8$\\
      65536 & 30 & 30 & 2.1\E$+0$ & 2.0\E$+0$ & 3.4\E$-1$ & 1.6\E$-8$\\
     131072 & 29 & 29 & 4.1\E$+0$ & 4.7\E$+0$ & 6.8\E$-1$ & 2.2\E$-8$\\
     \hline
    \end{tabular}
   }
  \end{center}
 \end{table}
 In 2D, both phases are very fast, easily beating the uncompressed LAPACK/ATLAS solver in both time ($\bigO (N^{3})$ and $\bigO (N^{2})$ for precomputation and solution, respectively) and memory, and coming quite close to the FMM/GMRES solver as well.
 \begin{table}
  \caption{Numerical results for solving Laplace's equation 3D at precision $\epsilon = 10^{-6}$. Parentheses denote extrapolated values.}
  \label{tab:solve-lap3d}
  \begin{center}
   {\footnotesize
    \begin{tabular}{r|rrcccc}
     \hline
     \multicolumn{1}{c|}{$N$} & \multicolumn{1}{c}{$K_{\row}$} & \multicolumn{1}{c}{$K_{\col}$} & $T_{\cm}$ & $T_{\qr}$ & $T_{\sv}$ & $E$\\
     \hline
       320 &  320 &  320 & 2.3\E$-1$ &  5.1\E$-2$  &  4.5\E$-3$ & 1.5\E$-10$\\
       720 &  628 &  669 & 1.1\E$+0$ &  1.6\E$-1$  &  9.3\E$-3$ & 5.1\E$-07$\\
      1280 &  890 &  913 & 3.7\E$+0$ &  5.0\E$-1$  &  3.2\E$-2$ & 1.0\E$-06$\\
      2880 & 1393 & 1400 & 1.7\E$+1$ &  1.9\E$+0$  &  9.1\E$-2$ & 1.2\E$-06$\\
      5120 & 1886 & 1850 & 4.7\E$+1$ &  6.0\E$+0$  &  1.6\E$-1$ & 2.2\E$-06$\\
     11520 & 2750 & 2754 & 1.4\E$+2$ & (1.9\E$+1$) & (3.9\E$-1$)\\
     20480 & 3592 & 3551 & 3.1\E$+2$ & (4.6\E$+1$) & (7.4\E$-1$)\\
     \hline
    \end{tabular}
   }
  \end{center}
 \end{table}
 The same is essentially true in 3D over the range of problem sizes tested, though it should be emphasized that FMM/GMRES has optimal $\bigO (N)$ complexity and so should prevail asymptotically. However, as observed previously \cite{greengard:2009:acta-numer,ho:2012:siam-j-sci-comput,martinsson:2005:j-comput-phys}, the solve time using recursive skeletonization following precomputation (comprising one application of $\mathbf{Q}^{*}$ and one backsolve with $\mathbf{R}$) is much faster than an individual FMM/GMRES solve: e.g., in 2D at $N = 131072$, $T_{\FMM} = 3.9$ s, while $T_{\sv} = 0.7$ s. This is significantly slower when compared with our UMFPACK-based sparse LU solver ($T_{\sv} \sim 0.1$ s) \cite{ho:2012:siam-j-sci-comput}. The difference may be due, in part, both to a higher constant inherent in the QR approach and to the overhead from interfacing with \MATLAB. Unfortunately, we were unable to perform the sparse QR factorizations in-core for the 3D case beyond $N \sim 10^{4}$; the corresponding data are extrapolated from the results of section \ref{sec:complexity}.

 \subsection{Least squares fitting of thin plate splines}
 \label{sec:thin-plate-spline}
 We next turned to an overdetermined problem involving 2D function interpolation using thin plate splines; see (\ref{eqn:thin-plate-spline}). More specifically, we sought to compute the coefficients $a_{j}$ of the interpolant
 $$
  g(\vec{x}) = \sum_{j = 1}^{N} a_{j} \phi_{\TPS} (\| \vec{x} - \vec{c}_{j} \|),
 $$
 that best matches a given function
 $$
  f(x, y) = \sin (4 \pi x) + \cos (2 \pi y) \sin (3 \pi xy) \quad \mbox{for $\vec{x} \equiv (x, y)$}
 $$
 in the least squares sense on some set of randomly chosen targets $\vec{x}_{i} \in [0, 1]^{2}$ for $i = 1, \dots, M$. The points $\vec{c}_{j}$ for $j = 1, \dots, N$ denote the centers of the splines and lie on a uniform tensor product grid on $[0, 1]^{2}$. This is an {\em inconsistent} linear system. Since the problem is somewhat ill-conditioned, we also add Tikhonov regularization with regularization parameter $\mu = 0.1$ as indicated in (\ref{eqn:overdetermined-regularized}).

 Timing results for various $M$ and $N$ at a fixed ratio of $M/N = 4$ with $\epsilon = 10^{-6}$ are shown in Fig.\ \ref{fig:overdetermined}, with detailed data in Table \ref{tab:overdetermined}.
 \begin{figure}
  \begin{center}
   \includegraphics{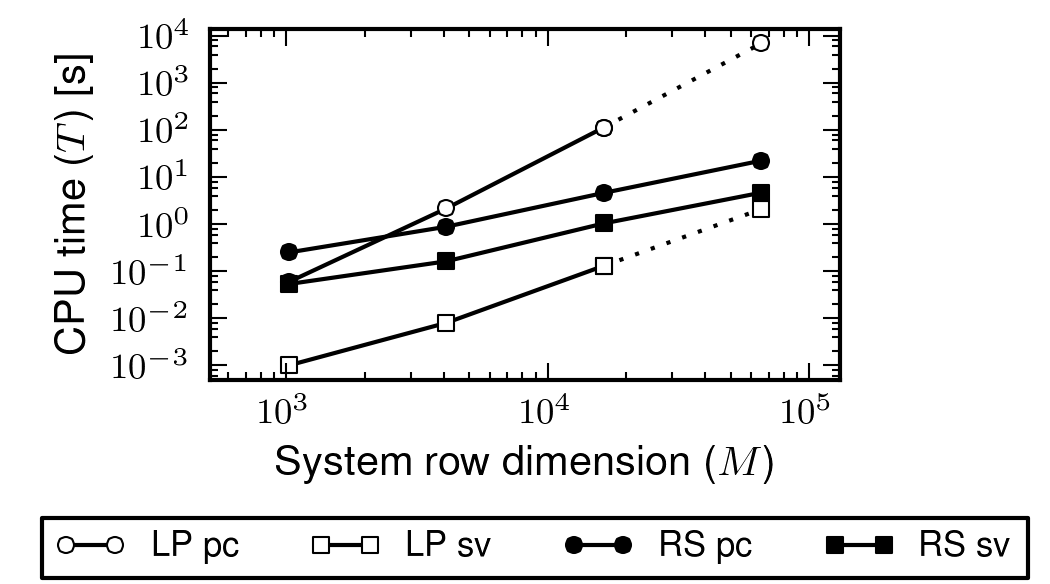}
  \end{center}
  \caption{CPU times for overdetermined thin plate spline fitting in 2D at precision $\epsilon = 10^{-6}$ using LAPACK/ATLAS, FMM/GMRES, and recursive skeletonization as a function of the system row dimension $M$, with the column dimension fixed proportionally at $N = M/4$; all other notation as in Fig.\ \ref{fig:solve-lap}.}
  \label{fig:overdetermined}
 \end{figure}
 The results are in line with our complexity estimates of $\bigO (M + N^{3/2})$ and $\bigO (M + N \log N)$ for precomputation and solution, respectively.
 \begin{table}
  \caption{Numerical results for overdetermined thin plate spline fitting in 2D at precision $\epsilon = 10^{-6}$.}
  \label{tab:overdetermined}
  \begin{center}
   {\footnotesize
    \begin{tabular}{rr|rrcccccc}
     \hline
     \multicolumn{1}{c}{$M$} & \multicolumn{1}{c|}{$N$} & \multicolumn{1}{c}{$K_{\row}$} & \multicolumn{1}{c}{$K_{\col}$} & $T_{\cm}$ & $T_{\qr}$ & $T_{\sv}$ & $n_{\iter}$ & $E$ & $R$\\
     \hline
      1024 &   256 & 174 & 148 & 7.7\E$-2$ & 1.7\E$-1$ & 5.3\E$-2$ & 1 & 4.1\E$-5$ & 1.4\E$-1$\\
      4096 &  1024 & 260 & 247 & 5.7\E$-1$ & 3.1\E$-1$ & 1.6\E$-1$ & 1 & 8.3\E$-5$ & 4.4\E$-2$\\
     16384 &  4096 & 399 & 391 & 3.1\E$+0$ & 1.5\E$+0$ & 1.0\E$+0$ & 1 & 3.9\E$-4$ & 1.6\E$-2$\\
     65536 & 16384 & 564 & 574 & 1.5\E$+1$ & 7.0\E$+0$ & 4.7\E$+0$ & 1 & 2.0\E$-6$ & 6.7\E$-3$\\
     \hline
    \end{tabular}
   }
  \end{center}
 \end{table}
 This compares favorably with the uncompressed complexities of $\bigO (M N^{2})$ and $\bigO (MN)$, respectively, for LAPACK/ATLAS. We also tested an iterative GMRES solver, which required from $33$ up to $130$ iterations on the largest problem considered using $A^{\trans}$ as a left preconditioner. Direct timings are unavailable since we did not have an FMM to apply the thin-plate spline kernel (or its transpose). The results are instead estimated using recursive skeletonization and an established benchmark FMM rate of about $10^{5}$ points per second in 2D. It is immediate that our fast solver outperforms FMM/GMRES due to the rapidly growing iteration count. Note also the convergence in the relative residual of roughly second order. In all cases, the deferred correction procedure converged with just one iteration.

 \subsection{Underdetermined charge fitting}
 We then considered an underdetermined problem: seeking a minimum-norm charge distribution in 2D. The setup is as follows. Let $\vec{x}_{j}$ for $j = 1, \dots, N$ be uniformly spaced points on the unit circle, each associated with a random charge $q_{j}$. We measure their induced field
 $$
  f(\vec{x}; q) = -\frac{1}{2 \pi} \sum_{j = 1}^{N} q_{j} \log \left\| \vec{x} - \vec{x}_{j} \right\|
 $$
 on a uniformly sampled outer ring of radius $1 + \delta$, and compute an equivalent set of charges $\tilde{q}_{j}$ with minimal Euclidean norm reproducing those measurements. Here, we set $\delta = 10^{-4}$ and sampled at $M = N/8$ observation points over a range of $N$.

 Timing results at precision $\epsilon = 10^{-9}$ are shown in Fig.\ \ref{fig:underdetermined}, with detailed data given in Table \ref{tab:underdetermined}.
 \begin{figure}
  \begin{center}
   \includegraphics{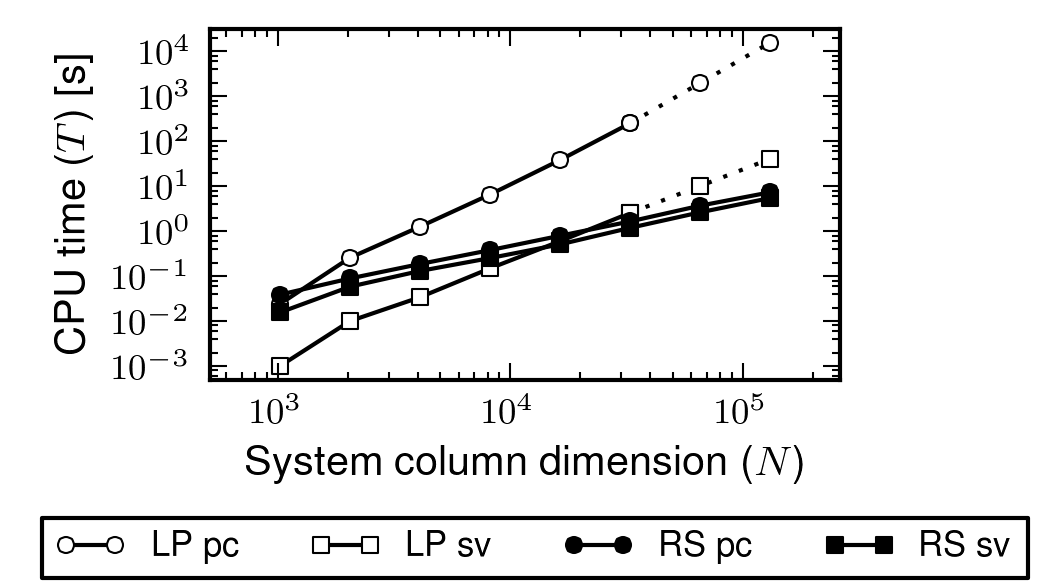}
  \end{center}
  \caption{CPU times for underdetermined charge fitting in 2D at precision $\epsilon = 10^{-9}$ using LAPACK/ATLAS, FMM/GMRES, and recursive skeletonization as a function of the system column dimension $N$, with the row dimension fixed proportionally at $M = N/8$; all other notation as in Fig.\ \ref{fig:solve-lap}.}
  \label{fig:underdetermined}
 \end{figure}
 Since the source and target points are separated (by an annular region of width $\delta$), our algorithm has optimal $\bigO (M + N)$ complexity, which is readily observed.
 \begin{table}
  \caption{Numerical results for underdetermined charge fitting in 2D at precision $\epsilon = 10^{-9}$.}
  \label{tab:underdetermined}
  \begin{center}
   {\footnotesize
    \begin{tabular}{rr|rrcccccc}
     \hline
     \multicolumn{1}{c}{$M$} & \multicolumn{1}{c|}{$N$} & \multicolumn{1}{c}{$K_{\row}$} & \multicolumn{1}{c}{$K_{\col}$} & $T_{\cm}$ & $T_{\qr}$ & $T_{\sv}$ & $n_{\iter}$ & $E$ & $R$\\
     \hline
       128 &   1024 &  69 &  72 & 1.7\E$-2$ & 2.1\E$-2$ & 1.6\E$-2$ & 1 & 1.6\E$-9$ & 1.6\E$-15$\\
       256 &   2048 &  80 &  80 & 4.4\E$-2$ & 4.3\E$-2$ & 5.8\E$-2$ & 2 & 1.3\E$-8$ & 1.1\E$-15$\\
       512 &   4096 &  89 &  90 & 9.6\E$-2$ & 8.7\E$-2$ & 1.3\E$-1$ & 2 & 6.1\E$-8$ & 1.9\E$-15$\\
      1024 &   8192 &  99 & 100 & 2.0\E$-1$ & 1.8\E$-1$ & 2.5\E$-1$ & 2 & 5.5\E$-8$ & 3.0\E$-15$\\
      2048 &  16384 & 108 & 110 & 4.0\E$-1$ & 3.8\E$-1$ & 5.1\E$-1$ & 2 & 3.6\E$-8$ & 1.8\E$-15$\\
      4096 &  32768 & 119 & 119 & 8.1\E$-1$ & 8.2\E$-1$ & 1.2\E$+0$ & 2 & 3.5\E$-8$ & 2.1\E$-15$\\
      8192 &  65536 & 128 & 131 & 1.6\E$+0$ & 2.0\E$+0$ & 2.6\E$+0$ & 2 & 4.8\E$-9$ & 7.1\E$-09$\\
     16384 & 131072 & 134 & 138 & 3.3\E$+0$ & 4.0\E$+0$ & 5.5\E$+0$ & 2 & 6.6\E$-9$ & 7.5\E$-09$\\
     \hline
    \end{tabular}
   }
  \end{center}
 \end{table}
 Furthermore, as we have solved an approximate, compressed system, we cannot in general fit the data exactly (with respect to the true operator). Indeed, we see relative residuals of order $\bigO (\epsilon)$ as predicted by the compression tolerance. Thus, our algorithm is especially suitable in the event that observations need to be matched only to a specified precision. Our semi-direct method vastly outperformed both LAPACK/ATLAS and FMM/GMRES, which required from $42$ up to $880$ iterations using $A^{\trans}$ as a right preconditioner. Deferred correction was successful in all cases within two steps.

 \subsection{Thin plate splines with updating}
 In our final example, we demonstrate the efficiency of our updating and downdating methods in the typical setting of fitting additional observations to an already specified overdetermined system. For this, we employed the thin plate spline approximation problem of section \ref{sec:thin-plate-spline} with $M = 16384$ and $N = 4096$, followed by the addition of $50$ new random target points. From section \ref{sec:updating}, the perturbed system can be written as (\ref{eqn:equality-least-squares-augmented}) with (\ref{eqn:row-addition}), i.e.,
 $$
  \mathbf{E} \mathbf{x} \simeq \mathbf{f} \quad \mbox{where $\mathbf{E} = \left[
  \begin{array}{c}
   A\phantom{_{+}}\\
   C^{\row}_{+}
  \end{array} \right]$}.
 $$

 We used GMRES with the left preconditioner $\mathbf{B} = [A^{+}, (C^{\row}_{+})^{*}]$, which, since $A$ has full column rank, gives $\mathbf{B} \mathbf{E} = I + (C^{\row}_{+})^{*} C^{\row}_{+}$, hence the preconditioned system is
 \begin{equation}
  \left[ I + (C^{\row}_{+})^{*} C^{\row}_{+} \right] \mathbf{x} \simeq \mathbf{B} \mathbf{f}.
  \label{eqn:update-setup}
 \end{equation}
 Note that only {\em one} application of $A^{+}$ is necessary, independent of the number of iterations required. Solving this in \MATLAB\ took $17$ iterations and a total of $1.9$ s, with $1.6$ s going towards setting up (\ref{eqn:update-setup}). The relative residual on the new data was $7.7 \times 10^{-3}$. This should be compared with the roughly $6$ s required to solve the problem without updating, treating it instead as a new system via our compressed algorithm (Table \ref{tab:overdetermined}). Although this difference is perhaps not very dramatic, it is worth emphasizing that the complexity here scales as $\bigO (M + N \log N)$ with updating versus $\bigO (M + N^{3/2})$ without, as the former needs only to apply $A^{+}$ while the latter needs also to compress and factor $\mathbf{E}$. Therefore, the asymptotics for updating are much improved.

 \section{Generalizations and conclusions}
 \label{sec:conclusion}
 In this paper, we have presented a fast semi-direct algorithm for over- and underdetermined least squares problems involving HBS matrices, and exhibited its efficiency and practical performance in a variety of situations including RBF interpolation and dynamic updating. In 1D (including boundary problems in 2D and problems with separated data in all dimensions), the solver achieves optimal $\bigO (M + N)$ complexity and is extremely fast, but it falters somewhat in higher dimensions, due primarily to the growing ranks of the compressed matrices as expressed by (\ref{eqn:rank-growth}). Developments for addressing this growth are now underway for square linear systems \cite{corona:2013:arxiv,ho:2013:arxiv}, and we expect these ideas to carry over to the present setting. Significantly, the term involving the larger matrix dimension is linear in all complexities (i.e., only $\bigO (M)$ instead of $\bigO (M N^{2})$ as for classical direct methods), which makes our algorithm ideally suited to large, rectangular systems where both $M$ and $N$ increase with refinement.

 {\em Remark}. If only {\em one} dimension is large so that the matrix is strongly rectangular, then standard methods are usually sufficient; see also \cite{meng:2014:siam-j-sci-comput,rokhlin:2008:proc-natl-acad-sci-usa,tygert:2009:arxiv}.

 Although we have not explicitly considered least squares problems with HBS equality constraints (we have only done so implicitly through our treatment of underdetermined systems), it is evident that our methods generalize. However, our complexity estimates can depend on the structure of the system matrix. In particular, if it is sparse, e.g., a diagonal weighting matrix, then our estimates are preserved. We can also, in principle, handle HBS least squares problems with HBS constraints simply by expanding out both matrices in sparse form.

 This flexibility is one of our method's main advantages, though it can also create some difficulties. In particular, the fundamental problem is no longer the unconstrained least squares system (\ref{eqn:overdetermined}) but the more complicated equality-constrained system (\ref{eqn:equality-least-squares}). Accordingly, more sophisticated iterative techniques \cite{barlow:1988:siam-j-numer-anal,barlow:1992:siam-j-numer-anal,van-loan:1985:siam-j-numer-anal} are used, but these can fail if the problem is too ill-conditioned. This is perhaps the greatest drawback of the proposed scheme. Still, our numerical results suggest that the algorithm remains effective for moderately ill-conditioned problems that are already quite challenging for standard iterative solvers. For severely ill-conditioned problems, other methods may be preferred.

 Finally, it is worth noting that fast direct solvers can be leveraged for other least squares techniques as well. This is straightforward for the normal equations, which are subject to well-known conditioning issues, and for the somewhat better behaved augmented system version \cite{arioli:1989:numer-math,bjorck:1996:siam}:
 $$
  \left[
  \begin{array}{cc}
   I & A\\
   A^{*}
  \end{array} \right] \left[
  \begin{array}{c}
   r\\
   x
  \end{array} \right] = \left[
  \begin{array}{c}
   b\\
   0
  \end{array} \right].
 $$
 This approach has the advantage of being immediately amenable to fast inversion techniques but at the cost of ``squaring'' and enlarging the system. Thus, all complexity estimates involve $M + N$ instead of $M$ and $N$ separately. In particular, the current generation of fast direct solvers would require, e.g., $\bigO ((M + N)^{3(1 - 1/d)})$ instead of $\bigO (M + N^{3(1 - 1/d)})$ work. With the development of a next generation of linear or nearly linear time solvers \cite{corona:2013:arxiv,hackbusch:2002:computing,ho:2013:arxiv}, this distinction may become less critical. Memory usage and high-performance computing hardware issues will also play important roles in determining which methods are most competitive. We expect these issues to become settled in the near future.

 \section*{Acknowledgements}
 We would like to thank the anonymous referees for their careful reading and insightful remarks, which have improved the paper tremendously.

\end{document}